\documentclass[letterpaper,11pt,reqno]{amsart}
\usepackage[portrait,margin=1in]{geometry}
\usepackage{comment}

\usepackage{mathrsfs,xfrac} 
\usepackage[colorlinks=true,linkcolor=blue,citecolor=blue,urlcolor=blue]{hyperref} 
\usepackage{amsmath,amssymb,amsthm,amsfonts,amsbsy,latexsym,dsfont,color,graphicx,enumitem}
\usepackage[foot]{amsaddr}
\usepackage{caption}
\usepackage{regexpatch}
\usepackage{comment}
\usepackage{todonotes}
\usepackage{comment}
\makeatletter
\xpatchcmd{\@todo}{\setkeys{todonotes}{#1}}{\setkeys{todonotes}{inline,#1}}{}{}
\makeatother

\newtheorem{thm}{Theorem}[section]
\newtheorem{lem}[thm]{Lemma}
\newtheorem{cor}[thm]{Corollary}
\newtheorem{prop}[thm]{Proposition}

\newtheorem{rem}[thm]{Remark}

\renewcommand{\le}{\leqslant}  
\renewcommand{\ge}{\geqslant} 
\renewcommand{\leq}{\leqslant}  
\renewcommand{\geq}{\geqslant} 

\newcommand{\ra}{\rangle}
\newcommand{\la}{\langle}

\newcommand{\eps}{\varepsilon}

\newcommand{\norm}[1]{\left\Vert#1\right\Vert}
\newcommand{\abs}[1]{\left\vert#1\right\vert}

 \let\gb=\beta

\newcommand{\cB}{\mathcal{B}}

\newcommand{\cN}{\mathcal{N}}\newcommand{\cO}{\mathcal{O}}

\newcommand{\mvgb}{\boldsymbol{\beta}}

\newcommand{\mvgr}{\boldsymbol{\rho}}
\newcommand{\mvgs}{\boldsymbol{\sigma}}\newcommand{\mvgt}{\boldsymbol{\tau}}

\newcommand{\bC}{\mathbb{C}}

\newcommand{\bN}{\mathbb{N}}

\newcommand{\dR}{\mathds{R}}

\DeclareMathOperator{\E}{\mathds{E}}
\DeclareMathOperator{\pr}{\mathds{P}}

\DeclareMathOperator{\var}{Var}

\DeclareMathOperator{\argmax}{argmax}

\DeclareMathOperator{\sech}{sech}

\usepackage{pgfplots}
\pgfplotsset{compat=1.18}
\usetikzlibrary{decorations.pathmorphing}

\begin{document}
\title[Fluctuation results in the SKFI model]{Fluctuations of the free energy of the Sherrington--Kirkpatrick model with ferromagnetic interaction}
\author[Dey]{Partha S.~Dey$^\star$}
\author[Kang]{Taegu Kang$^\dag$}
\date{\today}
\address{Department of Mathematics, University of Illinois Urbana--Champaign, 1305 W.~Green Street, Urbana, Illinois 61801}
\email{$^\star$psdey@illinois.edu, $^\dag$taeguk2@illinois.edu}

\subjclass[2020]{Primary: 60F05, 82B44; Secondary: 82D30.}
\keywords{Spin glass, Phase transition, Central limit theorem, Stein's method.}

\begin{abstract}
    We study the Sherrington--Kirkpatrick model with an additional Curie--Weiss ferromagnetic interaction (SKFI), whose phase diagram in the plane of the disorder strength $\beta$ and the ferromagnetic coupling $\gamma$ consists of a paramagnetic, a ferromagnetic, and a spin glass region. Our main result is a central limit theorem for the free energy in the ferromagnetic regime: at scale $N^{-1/2}$, the fluctuations coincide with those of the Sherrington--Kirkpatrick model in an effective external field determined by the limiting magnetization. We prove this for general mixed even $p$-spin interactions, using the fluctuation theory of Chen, Dey, and Panchenko~\cite{CDP17}. For the pure $2$-spin model we complete the picture of the phase diagram: we give a new proof of the order-$N^{-1}$ Gaussian fluctuations in the paramagnetic regime, obtained earlier by Banerjee~\cite{Banerjee20}, valid up to the critical window; we determine the limiting distribution on the critical line $\gamma = 1 + \theta N^{-1/2}$ separating the paramagnetic and ferromagnetic regimes; and we show that in the spin glass regime the fluctuations coincide with those of the zero-field SK model under a widely believed variance-divergence assumption. Our results form the Ising analogue of the fluctuation results of Baik and Lee~\cite{BL17} for the spherical SKFI model.
\end{abstract}
\maketitle
\setcounter{tocdepth}{1}\tableofcontents

%%%%%%%%%%%%%%%%%%%%%%%%%%%%%%%%
\section{Introduction}\label{sec:intro}
The Sherrington--Kirkpatrick (SK) model~\cite{SK75} is the canonical mean-field model of a spin glass. Its limiting free energy is given by the celebrated Parisi formula, predicted by Parisi~\cite{Par79,Par80} and proved rigorously by Talagrand and Panchenko~\cite{Talagrand06,Pan14}. Already in the original works~\cite{SK75,KS78}, the couplings were taken with a non-zero ferromagnetic mean, and the resulting phase diagram in the plane of disorder strength and ferromagnetic bias --- with a paramagnetic, a ferromagnetic, and a spin glass phase --- has been part of the physics of spin glasses ever since~\cite{Tou80,MPV87}. In this paper we study this model in the form of the SK Hamiltonian augmented by a Curie--Weiss interaction, which we call the SKFI (SK with ferromagnetic interaction) model. Quenched disorder and a uniform ferromagnetic drive compete: depending on their relative strengths, neither interaction produces macroscopic order (paramagnetic phase), the system develops non-zero macroscopic magnetization (ferromagnetic phase), or frustration produced by the disorder prevents uniform magnetization (spin glass phase). Equivalently, the Curie--Weiss term is the energy of a single stored pattern in the sense of Hopfield~\cite{Hop82}, so the SKFI model is also the SK model perturbed by one Hebbian pattern; see~\cite{BG98} for the Hopfield model literature.

At the level of the law of large numbers, the phase diagram of the SKFI model is understood: the limiting free energy is given by a variational formula over the magnetization, due to Chen~\cite{Chen14}, and the three phases correspond to the structure of its maximizers (Section~\ref{subsec:phase}). This paper concerns the next-order behavior --- the distributional fluctuations of the free energy --- where the phases are distinguished not only by the value of the limit but also by the scale of the fluctuations ($N^{-1}$, $N^{-1/2}$, or unknown) and by the shape of the limit law.

For the \emph{spherical} version of the model, in which the discrete hypercube is replaced by the sphere of radius $\sqrt N$ and whose phase diagram goes back to Kosterlitz, Thouless, and Jones~\cite{KTJ76}, Baik and Lee~\cite{BL17} obtained the limiting distribution of the free energy in each of the three phases, using an exact random matrix representation of the partition function. On the hypercube no such representation exists, and prior to this work only the paramagnetic phase had been analyzed~\cite{Banerjee20}. The main contribution of this paper is the limiting distribution of the free energy in the \emph{ferromagnetic} phase, which we establish in the generality of mixed even $p$-spin interactions (Theorem~\ref{thm:ferro}). Together with a new, short proof in the paramagnetic phase, a limit theorem on the critical line between the paramagnetic and ferromagnetic phases, and a conditional result in the spin glass phase, this yields the trichotomy stated as Theorem~\ref{thm:A} below --- the Ising analogue of~\cite[Theorem~1.4]{BL17}.

\subsection{The model}\label{subsec:model}
The (pure $2$-spin) SKFI model depends on an inverse temperature $\beta \ge 0$ and a ferromagnetic coupling constant $\gamma \ge 0$. For a spin configuration $\mvgs \in \Sigma_N := \{-1,+1\}^N$, the Hamiltonian is
\begin{align}
    H_N(\mvgs) = H_N^{SK}(\mvgs) + \frac{\gamma N}{2}\, m(\mvgs)^2,
    \qquad
    m(\mvgs) := \frac1N \sum_{i=1}^N \sigma_i,
    \label{eq:H_SKFI_2spn}
\end{align}
where the disordered part is
\begin{align*}
    H_N^{SK}(\mvgs) = \frac{\gb}{\sqrt N} \sum_{1 \le i < j \le N} g_{i,j}\, \sigma_i \sigma_j
\end{align*}
with $(g_{i,j})_{i<j}$ i.i.d.~standard Gaussian random variables. The partition function, the (random) free energy, and the Gibbs measure are
\begin{align*}
    Z_N(\gb,\gamma) := \sum_{\mvgs\in\Sigma_N} \exp H_N(\mvgs),
    \qquad
    F_N(\gb,\gamma) := \frac1N \ln Z_N(\gb,\gamma),
    \qquad
    G_N(\mvgs) := \frac{\exp H_N(\mvgs)}{Z_N(\gb,\gamma)}.
\end{align*}

We write $\bar F_N := \E F_N$ for the quenched-averaged free energy. For a function $f$ on $\Sigma_N^n$, the Gibbs average is
\begin{align*}
    \la f \ra := \sum_{\mvgs^1,\dots,\mvgs^n\in\Sigma_N} f(\mvgs^1,\dots,\mvgs^n)\, G_N(\mvgs^1)\cdots G_N(\mvgs^n).
\end{align*}

When the Hamiltonian is $H_N^{SK}(\mvgs) + hNm(\mvgs)$ for some $h\in\dR$, the model reduces to the classical SK model with external field $h$; we denote the corresponding quantities by $Z_N^{SK}(\gb,h)$, $F_N^{SK}(\gb,h)$, $\bar F_N^{SK}(\gb,h)$, $G_N^{SK}$, and $\la\cdot\ra^{SK} = \la\cdot\ra^{SK}_{(\beta,h)}$.

We work throughout Sections~\ref{sec:intro} and~\ref{sec:proof-mag}--\ref{sec:proof-crit} with the ``ordered-sum'' convention above, in which the disorder contains no diagonal or repeated terms. The general mixed even $p$-spin model, for which our ferromagnetic-regime result is proved, is defined with the ``full-sum'' convention in Section~\ref{sec:mixed}; Remark~\ref{rem:convention} makes the translation between the two conventions precise.

\subsection{Limiting free energy and the phase diagram}\label{subsec:phase}
The existence of the thermodynamic limit for the SK model was established by Guerra and Toninelli~\cite{GT02}
\begin{align*}
    F^{SK}(\beta,h) := \lim_{N\to\infty} \bar F_N^{SK}(\beta,h),
\end{align*}
and the limit is given by the Parisi formula~\cite{Talagrand06,Pan14} (see Section~\ref{subsec:SKfluct}). For the SKFI model, Chen~\cite{Chen14} proved that the limiting free energy satisfies the variational formula
\begin{align}
    F(\beta,\gamma)
    := \lim_{N\to\infty} \bar F_N(\beta,\gamma)
    = \max_{\mu\in[-1,1]} \Big\{ F^{SK}(\beta,\gamma\mu) - \frac{\gamma\mu^2}{2} \Big\},
    \label{eq:variational}
\end{align}
and we define the set of maximizers
\begin{align*}
    \Omega(\beta,\gamma) := \argmax_{\mu\in[-1,1]} \Big\{ F^{SK}(\beta,\gamma\mu) - \frac{\gamma\mu^2}{2} \Big\}.
\end{align*}
At $\gamma=0$ the variational functional is constant in $\mu$ and $\Omega(\beta,0)=[-1,1]$ degenerately; the model is then the zero-field SK model, and we accordingly restrict the phase-diagram discussion to $\gamma>0$. Since $F^{SK}(\beta,\cdot)$ is even, $\Omega(\beta,\gamma)$ is symmetric about $0$. The function $F$ is convex and locally Lipschitz on the parameter space. Hence, by Rademacher's theorem, differentiable at Lebesgue-almost-every $(\beta,\gamma)$; by~\cite[Proposition~2]{Chen14}, $\partial F/\partial\gamma$ exists at $(\beta,\gamma)$ if and only if $\abs{\Omega(\beta,\gamma)}=1$ or $\Omega(\beta,\gamma)=\{-\mu,\mu\}$ for some $0<\mu<1$. Thus for typical parameters the maximizer structure is as simple as symmetry permits. (For the infinite-dimensional parameter space of the mixed model the same conclusion holds in the sense of Aronszajn-null sets; see Remark~\ref{rem:aronszajn}.)

The maximizer set organizes the phase diagram into three regimes.
\begin{enumerate}
    \item \emph{Paramagnetic regime:} $\beta\in[0,1)$, $\gamma\in(0,1)$. Both interactions are weak and the system exhibits no macroscopic order; $\Omega(\beta,\gamma)=\{0\}$.
    \item \emph{Ferromagnetic regime:} $\Omega(\beta,\gamma)=\{-\mu,\mu\}$ for some $\mu>0$. The ferromagnetic coupling dominates, producing non-zero macroscopic magnetization.
    \item \emph{Spin glass regime:} $\beta>1$ and $\Omega(\beta,\gamma)=\{0\}$. Strong disorder induces frustration, preventing uniform magnetization while sustaining a complex landscape of equilibrium states.
\end{enumerate}
The following proposition locates the three regimes in the $(\beta,\gamma)$-plane; see Figure~\ref{fig:Ising_phase_diagram}. In particular, the spin glass regime contains $\{\beta>1,\ \gamma\le1\}$ and the ferromagnetic regime contains $\{\beta\le1,\ \gamma>1\}$. The regimes are separated by the critical lines $\{\gamma=1,\, 0\le\beta\le 1\}$ and $\{\beta=1,\, 0<\gamma\le 1\}$, on which $\Omega=\{0\}$ by part (i) below, and by the ferromagnet--spin glass boundary discussed in Remark~\ref{rem:FMSGboundary}; we leave these boundary lines unassigned to any regime. 
Theorem~\ref{thm:A}(ii) describes the free energy fluctuations across the critical line $\gamma=1$.

\begin{prop}\label{prop:magnetization}
The following statements hold:
\begin{enumerate}[label=\textup{\roman*)}]
    \item For any $\beta\ge0$ and $0<\gamma\le1$, we have $\Omega(\beta,\gamma)=\{0\}$.
    \item For any $\beta\le1$ and $\gamma>1$, we have $0\notin\Omega(\beta,\gamma)$.
    \item For $\gamma>1$, let $\theta=\theta(\gamma)>0$ be the unique positive solution of $\tanh(\gamma\theta)=\theta$, and let $q = q(\beta)$ be the largest solution in $[0,1)$ of $q = \E\tanh^2(\beta\sqrt q\, z)$, where $z$ is a standard Gaussian random variable; thus $q(\beta) = 0$ for $\beta\le1$ and $q(\beta) \in (0,1)$ is unique for $\beta>1$. Suppose
    \begin{align*}
        \ln\cosh(\gamma\theta) - \frac{\gamma\theta^2}{2} \;>\; \E\ln\cosh\big(\beta\sqrt q\, z\big) + \frac{\beta^2}{4}\,(1-q)^2 .
    \end{align*}
    Then $0\notin\Omega(\beta,\gamma)$. In particular, $0\notin\Omega(\beta,\gamma)$ whenever
    \begin{align*}
        \gamma \;>\; \min\Big\{ \frac{\beta^2}{2} + 2\ln2,\ \sqrt{\frac{2}{\pi}}\Big(2\beta + \frac1\beta\Big) \Big\}.
    \end{align*}
    \item For $\beta>1$ and $\gamma\ge\beta^2$, we have $0\notin\Omega(\beta,\gamma)$.
    \item For $\gamma>1$ and $\theta=\theta(\gamma)$ as in (iii), suppose
    \begin{align*}
        \ln\cosh(\gamma\theta) - \frac{\gamma\theta^2}{2} > \beta E_0,
        \qquad\text{where}\quad
        E_0 := \lim_{N\to\infty} \frac1N\, \E \max_{\mvgs\in\Sigma_N} \frac{1}{\sqrt N}\sum_{i<j} g_{ij}\sigma_i\sigma_j
    \end{align*}
    is the ground-state energy of the SK model. Then $0\notin\Omega(\beta,\gamma)$. Moreover, $E_0 \le \sqrt{\ln 2}$, so the same conclusion holds under the explicit condition $\ln\cosh(\gamma\theta) - \frac{\gamma\theta^2}{2} > \beta\sqrt{\ln2}$. 
\end{enumerate}
\end{prop}

\begin{rem}\label{rem:comparison}
    Parts (iii)--(v) cover complementary portions of the ferromagnetic phase at lower temperatures. At low temperature, (iv) requires $\gamma\ge\beta^2$, whereas (iii) and (v) are \emph{linear} in $\beta$: as $\beta\to\infty$, (iii) applies for $\gamma \ge 2\sqrt{2/\pi}\,\beta + O(1) \approx 1.60\,\beta$ and (v) for $\gamma \ge 2E_0\beta + O(1) \approx 1.53\,\beta$, matching the linear order of the conjectured boundary $\gamma=\beta$ up to a constant factor. Here the constant $E_0$ is characterized by the zero-temperature Parisi formula~\cite{AC17}, with numerical value $E_0 \approx 0.7632$~\cite{CR02}. Although (v) has the better slope, the boundary of (iii) lies below that of (v) until $\beta\approx20$, so only (iii) is displayed in Figure~\ref{fig:Ising_phase_diagram}. The replica symmetric bound subtracted in (iii) can be sharpened further by numerically optimizing Guerra's replica symmetry breaking bound~\cite{Gue03}; in the limit $\beta\to\infty$ the resulting value grows like $E_0\beta$, which is exactly what (v) captures.
\end{rem}

We remark that the physics literature further divides the magnetized region into a replica symmetric ferromagnet and a ``mixed'' phase with broken replica symmetry, separated by a de Almeida--Thouless-type line~\cite{AT78,Tou80,MPV87}. This distinction concerns only the structure of the Gibbs measure within each magnetized state, not the location of the maximizers of the free energy functional; hence it does not affect $\Omega$, and Theorem~\ref{thm:ferro} holds throughout the magnetized region, at any temperature. The exact location of the ferromagnet--spin glass boundary, by contrast, is a question about $\Omega$ itself, which Proposition~\ref{prop:magnetization} constrains but does not resolve; we discuss it in the following remark.

\begin{rem}[The ferromagnet--spin glass boundary]\label{rem:FMSGboundary}
    Let $\gamma_c(\beta) := \inf\{\gamma>0 : 0\notin\Omega(\beta,\gamma)\}$ denote the onset of ferromagnetic order. Proposition~\ref{prop:magnetization} gives $\gamma_c(\beta)=1$ for $\beta\le1$ and, for $\beta>1$, upper bounds of linear order in $\beta$; the predicted boundary is $\gamma_c(\beta)=\beta$~\cite{Tou80,MPV87}, as is rigorous in the spherical model~\cite{BL17}. The variational structure~\eqref{eq:variational} identifies $\gamma_c$ with a zero-field susceptibility of the SK model: writing $\varphi_\beta(h) := F^{SK}(\beta,h)-F^{SK}(\beta,0)$, a convexity argument shows
    \begin{align*}
        \gamma_c(\beta) = \inf_{h>0}\frac{h^2}{2\varphi_\beta(h)},
    \end{align*}
    so $\gamma_c(\beta)=\beta$ is equivalent to the pair of statements $\lim_{h\downarrow0}2\varphi_\beta(h)/h^2 = 1/\beta$ (Toulouse's marginality identity, an established consequence of the Parisi solution in the physics literature but open mathematically) and $\varphi_\beta(h)\le h^2/(2\beta)$ for all $h>0$. Note that at finite $N$ the susceptibility is trivial: gauge invariance of the disorder gives $\partial_h^2 \bar F_N^{SK}(\beta,h)|_{h=0} = 1$ for every $N$ and $\beta$, so the prediction asserts precisely that the limits $N\to\infty$ and $h\downarrow0$ fail to commute for $\beta>1$, replica symmetry breaking depressing the equilibrium susceptibility from $1$ to $1/\beta$. In the spherical model this susceptibility is an explicit random matrix quantity, which is why the boundary is a theorem there; on the hypercube it is an integral against the zero-field Parisi measure, whose structure for $\beta>1$ is not rigorously understood.
\end{rem}

\begin{figure}[htbp]
    \centering
    \begin{tikzpicture}
\begin{axis}[
    width=9.5cm,
    height=9.5cm,
    xlabel={$\beta$},
    ylabel={$\gamma$},
    xmin=0, xmax=6.5,
    ymin=0, ymax=6.5,
    axis lines=left,
    ticks=both,
]
% exact boundary of Proposition (iii)
\addplot[black, thick] coordinates {
(0.050,1.029) (0.100,1.060) (0.150,1.091) (0.200,1.124) (0.250,1.158) (0.300,1.193) (0.350,1.229) (0.400,1.266) (0.450,1.305) (0.500,1.345) (0.550,1.386) (0.600,1.429) (0.650,1.474) (0.700,1.520) (0.750,1.567) (0.800,1.616) (0.850,1.667) (0.900,1.720) (0.950,1.774) (1.000,1.830) (1.100,1.947) (1.200,2.070) (1.300,2.197) (1.400,2.328) (1.500,2.463) (1.600,2.599) (1.700,2.738) (1.800,2.879) (1.900,3.022) (2.000,3.166) (2.100,3.311) (2.200,3.458) (2.300,3.606) (2.400,3.754) (2.500,3.904) (2.600,4.054) (2.700,4.205) (2.800,4.356) (2.900,4.508) (3.000,4.661) (3.100,4.814) (3.200,4.967) (3.300,5.121) (3.400,5.275) (3.500,5.430) (3.600,5.585) (3.700,5.740) (3.800,5.895) (3.900,6.051) (4.000,6.206) (4.100,6.362) (4.200,6.518) (4.300,6.675) (4.400,6.831) (4.500,6.988) (4.600,7.144) (4.700,7.301) (4.800,7.458) (4.900,7.615) (5.000,7.772) (5.100,7.930) (5.200,8.087) (5.300,8.245) (5.400,8.402) (5.500,8.560)
};
% explicit sufficient condition in Proposition (iii)
\addplot[gray, dotted, thick, domain=0.35:6.4, samples=200]
    {min(x*x/2 + 2*ln(2), sqrt(2/pi)*(2*x + 1/x))};
% boundary from Proposition (iv)
\addplot[red, dashed, thick, domain=1:2.87, samples=100] {x*x};

% beta=1, gamma=1
\addplot[black, thick] coordinates {(1,0) (1,1)};
\addplot[black, thick] coordinates {(0,1) (1,1)};

% conjectured boundary gamma = beta
\draw[decorate, decoration={snake, segment length=8mm, amplitude=0.55mm}]
    (axis cs:1,1) -- (axis cs:6.4,6.4);
\node[fill=white, inner sep=1.5pt] at (axis cs:3.7,3.7) {?};

% labels
\node at (axis cs:0.5,0.5) {PM};
\node at (axis cs:0.8,4.5) {FM};
\node at (axis cs:4.5,1.2) {Spin Glass};
\node[rotate=77] at (axis cs:2.1,5.5) {\footnotesize (iv), $\gamma=\beta^2$};
\node[rotate=57] at (axis cs:3.5,5) {\footnotesize (iii)};
\node[rotate=57, gray] at (axis cs:3.1,5.6) {\footnotesize (iii), explicit};
\node[rotate=45] at (axis cs:5.35,4.75) {\footnotesize $\gamma=\beta$ (conjectural)};
\end{axis}
\end{tikzpicture}
    \caption{Phase diagram of the SKFI model. The solid curve is the boundary of the region given by Proposition~\ref{prop:magnetization}(iii), and the gray dotted curve is the explicit sufficient condition of (iii); both have asymptotic slope $2\sqrt{2/\pi}\approx1.60$. The red dashed curve is the line $\gamma=\beta^2$ of part (iv). Above any of these curves, $0\notin\Omega(\beta,\gamma)$. The boundary of part (v), of asymptotic slope $2E_0\approx1.53$, lies slightly above the solid curve in the displayed window and is omitted; see Remark~\ref{rem:comparison}. The wavy line is the conjectured boundary $\gamma=\beta$ between the ferromagnetic and spin glass regimes; see Remark~\ref{rem:FMSGboundary}.}
    \label{fig:Ising_phase_diagram}
\end{figure}
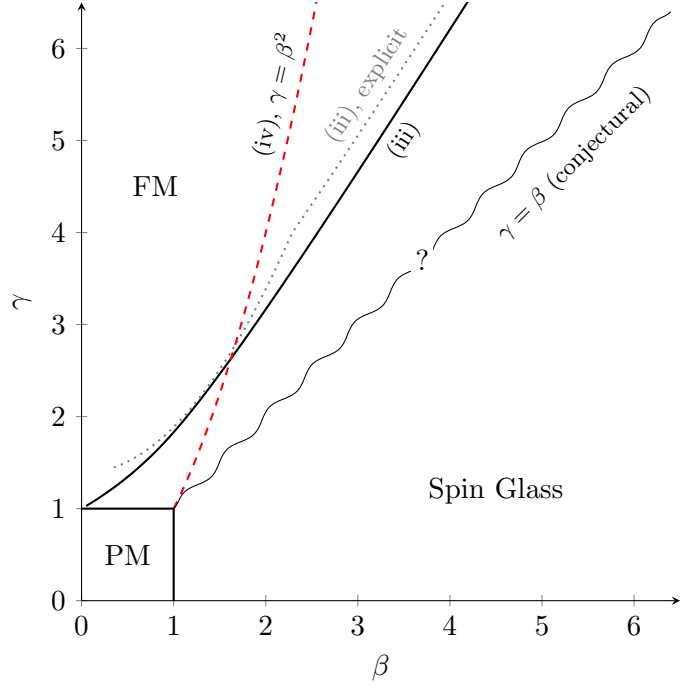

\subsection{Main results: fluctuations across the phase diagram}\label{subsec:main}
Our main result determines the limiting distribution of the free energy in each regime accessible to current techniques: the paramagnetic regime, the critical line $\gamma \approx 1$ separating it from the ferromagnetic regime, and the ferromagnetic regime. 

\begin{thm}\label{thm:A}
\begin{enumerate}[label=\textup{\roman*)}]
    \item \textup{(Paramagnetic regime.)} Let $\beta<1$ and let $(\gamma_N)_{N\ge1}$ satisfy $0\le\gamma_N<1$ and $\sqrt N(1-\gamma_N)\to\infty$. Then
    \begin{align*}
        N\Big(F_N(\beta,\gamma_N) - \frac{\beta^2}{4} - \ln 2\Big) + \frac12\ln(1-\gamma_N)
        \overset{d}{\longrightarrow}
        \cN\Big( \frac14\ln(1-\beta^2),\; -\frac12\ln(1-\beta^2) - \frac{\beta^2}{2} \Big).
    \end{align*}
    \item \textup{(Critical line.)} Let $\beta<1$ and $\gamma_N = 1 + \theta N^{-1/2}$ for some $\theta\in\dR$. Then
    \begin{align*}
        &N F_N(\beta,\gamma_N) - N\ln 2 - \frac{N-1}{4}\beta^2 - \frac14\ln N\\
        &\qquad\overset{d}{\longrightarrow}
        \zeta_1 + \ln \int_{\dR} \exp\Big( \Big(\zeta_2 + \frac{\theta}{2}\Big)y^2 - \Big(\frac1{12} + \frac{\alpha_2^2}{2}\Big)y^4 \Big)\,\frac{dy}{\sqrt{2\pi}},
    \end{align*}
    where $\alpha_1 = \big( -\tfrac12\ln(1-\beta^2) - \tfrac{\beta^2}{2} \big)^{1/2}$, $\alpha_2 = \big( \tfrac{\beta^2}{2(1-\beta^2)} \big)^{1/2}$, and $\zeta_1\sim\cN(-\alpha_1^2/2,\,\alpha_1^2)$ and $\zeta_2\sim\cN(0,\alpha_2^2)$ are independent.
    \item \textup{(Ferromagnetic regime.)} Suppose $\Omega(\beta,\gamma) = \{-\mu,\mu\}$ for some $\mu>0$. Then
    \begin{align*}
        \sqrt N\,\big( F_N(\beta,\gamma) - \bar F_N(\beta,\gamma) \big)
        \overset{d}{\longrightarrow}
        \cN(0,\nu),
    \end{align*}
    where $\nu = \nu(\beta,\gamma\mu) > 0$ is the variance functional of the SK model with external field $\gamma\mu$, defined in~\eqref{eq:nu} below with $\beta_2=\beta/\sqrt2$, $\beta_p=0$ ($p\neq2$), $h=\gamma \mu$ .
\end{enumerate}
\end{thm}

Part (iii) is a special case of Theorem~\ref{thm:ferro}, which is stated and proved for general mixed even $p$-spin interactions and identifies the fluctuations with those of the SK model in the effective external field $\gamma\mu$: the quadratic Curie--Weiss term $\frac{\gamma N}{2}m^2$ linearizes around the concentrated magnetization $\pm\mu$. Parts (i) and (ii) are proved in Sections~\ref{sec:proof-para} and~\ref{sec:proof-crit}. In part (i), the constant-order shift $\frac12\ln(1-\gamma_N)$ is the only trace the ferromagnetic interaction leaves at scale $N^{-1}$; the result recovers, and its proof simplifies, the paramagnetic-regime analysis of Banerjee~\cite{Banerjee20}, reducing it to the classical CLT of Aizenman, Lebowitz, and Ruelle~\cite{ALR87} for the zero-field SK model,
\begin{align}\label{eq:beta<1_dist_lim}
    N\Big( F_N^{SK}(\beta,0) - \frac{\beta^2}{4} - \ln 2 \Big)
    \overset{d}{\longrightarrow}
    \cN\Big( \frac14\ln(1-\beta^2),\; -\frac12\ln(1-\beta^2) - \frac{\beta^2}{2} \Big),
\end{align}
together with the recent quantitative control of small external fields by Dey and Wu~\cite{DW23}. Part (ii) is in fact universal: in Section~\ref{sec:proof-crit} we prove it for i.i.d.~symmetric couplings with unit variance and all moments finite (Theorem~\ref{thm:g=1_general}), the limit law depending on the coupling distribution only through $\E J^4$, consistent with the universality of the SK free energy~\cite{CH06}.

In the spin glass regime the exact fluctuations of the zero-field SK free energy are themselves unknown for $\beta>1$; the following unconditional comparison shows that whatever they are, the SKFI model with $\gamma<1$ shares them, provided the variance diverges --- a property universally accepted in the physics literature (see, e.g.,~\cite{CPSV92,Asp08}, and Section~\ref{subsec:related} for the known rigorous bounds).

\begin{prop}\label{prop:SG}
Let $\beta\ge0$ and $0\le\gamma<1$. Then $\ln Z_N(\beta,\gamma) \ge \ln Z_N^{SK}(\beta,0)$ almost surely, and for every $t\ge0$ and every $N$,
\begin{align*}
    \pr\Big( \ln Z_N(\beta,\gamma) - \ln Z_N^{SK}(\beta,0) \ge t \Big) \le \frac{e^{-t}}{\sqrt{1-\gamma}}.
\end{align*}
In particular, if $\nu_N := \var\big(\ln Z_N^{SK}(\beta,0)\big) \to \infty$, then
\begin{align*}
    \frac{\ln Z_N(\beta,\gamma) - \ln Z_N^{SK}(\beta,0)}{\sqrt{\nu_N}} \overset{p}{\longrightarrow} 0.
\end{align*}
\end{prop}

Thus, whenever the zero-field SK free energy admits a distributional limit at a diverging scale $\sqrt{\nu_N}$, the SKFI free energy with any fixed $\gamma<1$ has the same limit under the same centering and normalization.

\begin{rem}[Boundary between the paramagnetic and spin glass regimes]\label{rem:PMSGboundary}
    When $\beta_N<1$ with $\beta_N\to1$ and $\inf_N N^{1/3}(1-\beta_N^2) >0$, our companion paper~\cite{DK26} proves that
    \begin{align*}
        N\cdot \frac{F_N^{SK}(\beta_N,0) - \bar F_N^{SK}(\beta_N,0)}{\sqrt{\nu(\beta_N)}} \overset{d}{\longrightarrow} \cN(0,1),
        \qquad \nu(\beta) := -\frac12\ln(1-\beta^2) - \frac{\beta^2}{2},
    \end{align*}
    and at the critical temperature itself, Du and Huang~\cite{DH26} prove that
    \begin{align*}
        N\cdot \frac{F_N^{SK}(1,0) - \bar F_N^{SK}(1,0)}{\sqrt{(\ln N)/6}} \overset{d}{\longrightarrow} \cN(0,1).
    \end{align*}
    Since the normalization diverges in either case, the argument proving Proposition~\ref{prop:SG} applies verbatim: for any fixed $\gamma<1$, the SKFI free energy $F_N(\beta_N,\gamma)$ satisfies the same distributional limits, with the same centerings and normalizations.
\end{rem}

\subsection{Related results}\label{subsec:related}
\textbf{Spherical SKFI model.} Baik and Lee~\cite{BL17} analyzed the spherical analogue of~\eqref{eq:H_SKFI_2spn}, with $\mvgs \in S_{N-1} := \{\mvgs\in\dR^N : \norm{\mvgs}^2 = N\}$ and partition function
\begin{align*}
    Z_N^{SSKFI} := \int_{S_{N-1}} \exp H_N(\mvgs)\, d\omega_N(\mvgs),
\end{align*}
$\omega_N$ the uniform measure on $S_{N-1}$. The spherical relaxation admits an exact contour-integral representation in terms of the spectrum of the disorder matrix, so that random matrix theory applies; the limiting free energy and the phase diagram in Figure~\ref{fig:Spherical_phase_diagram} were first derived by Kosterlitz, Thouless, and Jones~\cite{KTJ76} and made rigorous in~\cite{BL17}. The spherical theory has since been refined in directions that parallel our results: the ferromagnet--paramagnet critical window was analyzed in~\cite{BLW18}, and the spin-glass--paramagnet window and the triple point in~\cite{JKOP24}. Their fluctuation results, alongside ours, are summarized in Table~\ref{tab:comparison} and Figure~\ref{fig:Spherical_phase_diagram}. The qualitative picture is identical, with one structural difference: in the spherical model the spin glass fluctuations are driven by the largest eigenvalue and are Tracy--Widom of order $N^{-2/3}$, and the ferromagnet--spin glass boundary is the explicit BBP transition line $\beta=\gamma$; on the hypercube, both the spin glass fluctuations and the exact boundary remain open (Remark~\ref{rem:FMSGboundary}).

\begin{table}[htbp]
\centering
\begin{tabular}{l@{ }l@{ }l}
\hline
Regime & Spherical SKFI~\cite{BL17} & SKFI (this paper) \\
\hline
Paramagnetic & Gaussian, $O(N^{-1})$ & Gaussian, $O(N^{-1})$ \hfill [Thm~\ref{thm:A}(i)] \\
PM--FM critical line & Gaussian--TW crossover~\cite{BLW18} & non-Gaussian, $O(N^{-1})$ \hfill [Thm~\ref{thm:A}(ii)] \\
Ferromagnetic & Gaussian, $O(N^{-1/2})$ & Gaussian, $O(N^{-1/2})$ \hfill [Thm~\ref{thm:A}(iii)] \\
Spin glass & $\mathrm{TW}_1$, $O(N^{-2/3})$ & $=$ zero-field SK, conditional \hfill [Prop~\ref{prop:SG}] \\
FM--SG boundary & $\gamma=\beta$, rigorous & linear order proved \hfill [Prop~\ref{prop:magnetization}(iii),(v)] \\
\hline
\end{tabular}
\medskip
\caption{Free energy fluctuations of the spherical and Ising SKFI models.}
\label{tab:comparison}
\end{table}

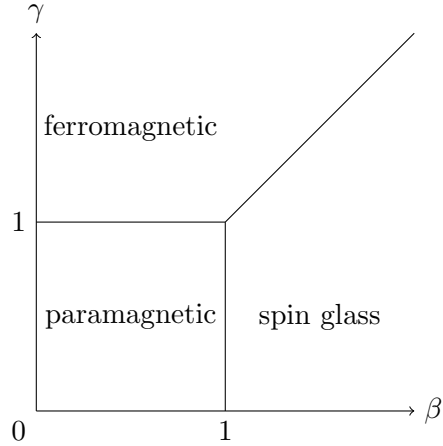
\begin{figure}[htbp]
    \centering
    \begin{tikzpicture}[scale=2.5]
    \draw[->] (0,0) -- (2,0) node[right] {$\beta$}; \draw[->] (0,0) -- (0,2) node[above] {$\gamma$};
    \draw (0,1) -- (1,1); \draw (1,0) -- (1,1); \draw (1,1) -- (2,2);
    \node at (0.5,0.5) {paramagnetic}; \node at (0.5,1.5) {ferromagnetic}; \node at (1.5,0.5) {spin glass};
    \node[below left] at (0,0) {0}; \node[below] at (1,0) {1}; \node[left] at (0,1) {1};
    \end{tikzpicture}
    \caption{Phase diagram of the spherical SKFI model~\cite{BL17}. The boundaries are the explicit lines $\beta=1$, $\gamma=1$, and $\beta=\gamma$.}
    \label{fig:Spherical_phase_diagram}
\end{figure}

\textbf{SK model without external field.} In the high-temperature region $\beta<1$, the CLT~\eqref{eq:beta<1_dist_lim} of~\cite{ALR87} is the input to Theorem~\ref{thm:A}(i); see~\cite{CN95} for an alternative proof via stochastic calculus. At the critical temperature $\beta=1$, the variance was bounded by $O(\ln^2 N)$ in~\cite{CL19}, and the physics prediction $\var\big(\ln Z_N^{SK}(1,0)\big) = \frac16\ln N + O(1)$ of~\cite{Asp08} was recently confirmed, together with a Gaussian CLT, in~\cite{DH26}; the CLT extends to the near-critical window $1-\beta_N^2 \gtrsim N^{-1/3}$ by our companion paper~\cite{DK26}. For $\beta>1$, even the order of the fluctuations is a major open problem: the best known upper bound is $\var\big(\ln Z_N^{SK}(\beta,0)\big) = O(N/\ln N)$~\cite{Cha09}, while the divergence of the variance --- the hypothesis $\nu_N\to\infty$ of Proposition~\ref{prop:SG} --- is universally expected but unproven. This is why Proposition~\ref{prop:SG} is necessarily comparative.

\textbf{SK model with external field.} For the mixed even $p$-spin model (see Section~\ref{subsec:mixedmodel}) with Hamiltonian $H_N^{mix}(\mvgs) + hNm(\mvgs)$ and $h \ne 0$, Chen, Dey, and Panchenko~\cite{CDP17} proved that the free energy has Gaussian fluctuations of order $N^{-1/2}$, at every temperature. This is the limit law behind Theorem~\ref{thm:A}(iii): in the ferromagnetic regime the magnetization concentrates on $\{\pm\mu\}$, the quadratic interaction $\frac{\gamma N}{2}m^2$ acts as the linear field $\gamma\mu\,Nm$, and Theorem~\ref{thm:ferro} identifies the SKFI fluctuations with those of the SK model in this effective field. The precise statements are recalled in Section~\ref{subsec:SKfluct}.

\textbf{Paramagnetic regime.} Banerjee~\cite{Banerjee20} first established the order-$N^{-1}$ Gaussian fluctuations in the paramagnetic regime, for fixed $\beta,\gamma<1$. We give a short independent proof of Theorem~\ref{thm:A}(i), reducing the problem to the zero-field CLT of~\cite{ALR87} via the quantitative small-field estimates of~\cite{DW23}; the argument moreover allows the coupling to approach the critical value, $\gamma_N\uparrow1$ with $\sqrt N(1-\gamma_N)\to\infty$.

\begin{rem}[The triple point]\label{rem:triplepoint}
    The proof of Theorem~\ref{thm:A}(ii) suggests that the paramagnet--ferromagnet crossover window widens as $\beta_N\uparrow1$: writing $\varepsilon_N := 1-\beta_N^2$, the relevant coupling window is expected to be $\gamma_N - 1 \asymp (N\varepsilon_N)^{-1/2}$, growing from $N^{-1/2}$ at fixed $\beta<1$ to $N^{-1/3}$ as $\varepsilon_N\downarrow N^{-1/3}$, where it meets the critical temperature window of~\cite{DK26,DH26}; the two transitions would then merge in the joint critical window $\beta = 1+O(N^{-1/3}\sqrt{\ln N})$, $\gamma = 1+O(N^{-1/3})$, in which the deformed matrix $\beta W + \frac{\gamma}{N}\mathbf{1}\mathbf{1}^{\mathsf T}$ is BBP-critical~\cite{BBP05}. On the subcritical side of this corner, we conjecture the following extension of Theorem~\ref{thm:A}(ii): if $\varepsilon_N\to0$ with $N^{1/3-\delta}\varepsilon_N\to\infty$ for some $\delta>0$, and $\gamma_N = 1 + w(N\varepsilon_N)^{-1/2}$ for fixed $w\in\dR$, then
    \begin{align*}
        &\ln Z_N(\beta_N,\gamma_N) - \ln Z_N^{SK}(\beta_N,0) - \frac14\ln\big(2N\varepsilon_N\big) + \frac12\ln(2\pi)\\
        &\qquad\overset{d}{\longrightarrow}\;
        \ln \int_\dR \exp\Big( \Big(\zeta + \frac{w}{\sqrt2}\Big)u^2 - \frac{u^4}{2} \Big)\, du,
    \end{align*}
    where $\zeta \sim \cN(0,1)$.
    Rescaling $y = (2\varepsilon/\beta^2)^{1/4}u$ in Theorem~\ref{thm:A}(ii) and letting $\varepsilon\downarrow0$ with $\theta = w\varepsilon^{-1/2}$ recovers exactly this limit, including the centering, so it is the unique law compatible with Theorem~\ref{thm:A}(ii) on overlapping scales. At the corner $\varepsilon_N\asymp N^{-1/3}$ and beyond, one expects a sphere-to-cube comparison in the spirit of~\cite{DK26,DH26}, with the spherical triple-point input of~\cite{JKOP24}, to be required.
\end{rem}

\subsection{Organization and notation}\label{subsec:organization}
Section~\ref{sec:mixed} introduces the mixed even $p$-spin SKFI model, recalls the fluctuation theory of the SK model with external field from~\cite{CDP17}, and states our main ferromagnetic-regime result, Theorem~\ref{thm:ferro}, together with the deduction of Theorem~\ref{thm:A}(iii). Section~\ref{sec:proof-mag} proves Proposition~\ref{prop:magnetization}. Section~\ref{sec:proof-para} proves Theorem~\ref{thm:A}(i) and Proposition~\ref{prop:SG}. Section~\ref{sec:proof-crit} proves Theorem~\ref{thm:A}(ii), in the universal form of Theorem~\ref{thm:g=1_general} allowing arbitrary symmetric disorder with finite moments. Section~\ref{sec:proof-ferro} proves Theorem~\ref{thm:ferro} and its corollary.

%%%%%%%%%%%%%%%%%%%%%%%%%%%%%%%%
\section{The ferromagnetic regime for mixed even-spin models}\label{sec:mixed}

In this section we state and discuss our main result: a central limit theorem for the free energy of the SKFI model in the ferromagnetic regime, for general mixed even $p$-spin interactions. Section~\ref{subsec:mixedmodel} defines the model and reconciles the two disorder conventions; Section~\ref{subsec:SKfluct} recalls the fluctuation theory of the SK model with external field~\cite{CDP17}, which is both the source of the limiting law and the target to which our theorem transports the SKFI free energy; Section~\ref{subsec:thmB} states Theorem~\ref{thm:ferro} and deduces Theorem~\ref{thm:A}(iii).

\subsection{The mixed even-spin SKFI model}\label{subsec:mixedmodel}
The mixed even-spin model is parameterized by inverse temperatures
\begin{align*}
    \mvgb \in \cB := \Big\{ (\beta_p)_{p\ge1} : \beta_p = 0 \text{ for all odd } p \text{ and } \sum_{p\ge1} 2^p\beta_p^2 < \infty \Big\},
\end{align*}
a ferromagnetic coupling $\gamma\ge0$, and an external field $h\in\dR$. For $\mvgs\in\Sigma_N$ the Hamiltonian is
\begin{align}\label{eq:H_SKFI_mix}
    H_N(\mvgs) = H_N^{mix}(\mvgs) + hN m(\mvgs) + \frac{\gamma N}{2} m(\mvgs)^2,
\end{align}
where the disordered part uses the full-sum convention,
\begin{align*}
    H_N^{mix}(\mvgs) = \sum_{p\ge1} \frac{\gb_p}{N^{(p-1)/2}} \sum_{1\le i_1,\dots,i_{p}\le N} g_{i_1,\dots,i_{p}}\, \sigma_{i_1}\cdots\sigma_{i_{p}}
\end{align*}
with i.i.d.~standard Gaussian coefficients $(g_{i_1,\dots,i_{p}})$. Equivalently, $H_N^{mix}$ is the centered Gaussian process on $\Sigma_N$ with covariance
\begin{align*}
    \E\, H_N^{mix}(\mvgs^1) H_N^{mix}(\mvgs^2) = N\, \xi(R_{\mvgs^1,\mvgs^2}),
    \qquad
    \xi(x) := \sum_{p\ge1} \beta_p^2 x^{p},
    \qquad
    R_{\mvgs,\mvgt} := \frac1N\sum_{i=1}^N \sigma_i\tau_i .
\end{align*}
The partition function $Z_N(\mvgb,\gamma,h)$, free energy $F_N(\mvgb,\gamma,h) := N^{-1}\ln Z_N(\mvgb,\gamma,h)$, mean free energy $\bar F_N := \E F_N$, Gibbs measure $G_N$, and Gibbs average $\la\cdot\ra$ are defined exactly as in Section~\ref{subsec:model}. When $\gamma=0$ the model is the classical mixed SK model, with quantities $Z_N^{SK}(\mvgb,h)$, $F_N^{SK}(\mvgb,h)$, $\la\cdot\ra^{SK}$, and Parisi limit $F^{SK}(\mvgb,h)$~\cite{Talagrand06,Pan14}. Chen's variational formula~\cite{Chen14} holds in this generality:
\begin{align*}
    F(\mvgb,\gamma,h) := \lim_{N\to\infty}\bar F_N(\mvgb,\gamma,h)
    = \max_{\mu\in[-1,1]}\Big\{ F^{SK}(\mvgb,\gamma\mu+h) - \frac{\gamma\mu^2}{2} \Big\},
\end{align*}
with maximizer set $\Omega(\mvgb,\gamma,h)$.

\begin{rem}[Differentiability in infinite dimensions]\label{rem:aronszajn}
    The parameter space $\cB\times\dR_{\ge0}\times\dR$ is a separable Banach space and $F$ is convex and locally Lipschitz on it; by Aronszajn's theorem~\cite{Aro76}, the set of points where $F$ fails to be G\^ateaux-differentiable is an Aronszajn-null set. Combined with~\cite[Proposition~2]{Chen14}, for typical parameters $\abs{\Omega(\mvgb,\gamma,h)}=1$ or $\Omega(\mvgb,\gamma,h)=\{-\mu,\mu\}$ with $0<\mu<1$.
\end{rem}

\begin{rem}[The two disorder conventions]\label{rem:convention}
    The full-sum Hamiltonian above is the convention of~\cite{Chen14,CDP17}, which we adopt throughout this section and Section~\ref{sec:proof-ferro}. In the pure $2$-spin case ($\beta_2=\beta/\sqrt2$, $\beta_p=0$ for $p\neq2$), it differs in law from the ordered-sum Hamiltonian of Section~\ref{subsec:model} only by the diagonal term $X_N := \beta_2 N^{-1/2}\sum_i g_{ii} \sim \cN(0,\beta^2/2)$, a single Gaussian independent of everything else. Since $X_N = O_P(1)$, the two conventions share the same limiting free energy and maximizer set $\Omega$, and any fluctuation result at scale $N^{1/2}$ transfers verbatim between them.
\end{rem}

\subsection{Fluctuations of the free energy in the SK model with external field}\label{subsec:SKfluct}
We recall the objects from~\cite{Chen13,CDP17} needed to state our result. For $\mvgb\in\cB$ and $h\in\dR$, the Parisi formula~\cite{Talagrand06,Pan14} expresses the limit as
\begin{align}\label{eq:Parisi_formula}
    F^{SK}(\mvgb,h) = \min_{\mu} \Big( \ln 2 + \E\,\Phi_\mu(0, h) - \frac12\int_0^1 \xi''(q)\, q\, \mu([0,q])\,dq \Big),
\end{align}
where the minimum is over probability measures $\mu$ on $[0,1]$ and $\Phi_\mu(q,x)$ solves the Parisi PDE
\begin{align*}
    \partial_q \Phi_\mu = -\frac{\xi''(q)}{2}\Big( \partial_{xx}\Phi_\mu + \mu([0,q])\,(\partial_x\Phi_\mu)^2 \Big),
    \qquad q\in[0,1],\; x\in\dR,
\end{align*}
with terminal condition $\Phi_\mu(1,x) = \ln\cosh x$. The minimizing Parisi measure $\mu_P$ is unique~\cite{AC15}, and when $h\ne0$ its support is bounded away from zero: $d := \min\operatorname{supp}\mu_P > 0$~\cite{Talagrand11}.

For $t\in(0,1)$ define
\begin{align}
    \varphi_t(s) := \E\big[ \partial_x\Phi_{\mu_P}(d, h+\chi_t^1(s))\; \partial_x\Phi_{\mu_P}(d, h+\chi_t^2(s)) \big],
    \qquad s\in[0,d],
    \label{eq:u_t}
\end{align}
where $(\chi_t^1(s),\chi_t^2(s))$ are jointly Gaussian with
\begin{align*}
    \E[\chi_t^1(s)^2] = \E[\chi_t^2(s)^2] = \xi'(d),
    \qquad
    \E[\chi_t^1(s)\chi_t^2(s)] = t\,\xi'(s).
\end{align*}
It was shown in~\cite{Chen13} that $\varphi_t$ has a unique fixed point $u_t\in[0,d]$, with $u_1 = d$, and we set
\begin{align}
    \nu = \nu(\mvgb,h) := \int_0^1 \xi(u_t)\,dt .
    \label{eq:nu}
\end{align}

The fluctuation theorem of Chen, Dey, and Panchenko then reads as follows.

\begin{lem}[{\cite[Theorem~2]{CDP17}}]\label{lem:SK_dist_limit}
    For $\mvgb\in\cB$ and $h\in\dR$ with $h\ne0$,
    \begin{align*}
        \lim_{N\to\infty} d_{TV}\Big( \frac{F_N^{SK}(\mvgb,h) - \bar F_N^{SK}(\mvgb,h)}{\sqrt{\nu/N}},\; g \Big) = 0,
    \end{align*}
    where $g$ is a standard Gaussian random variable and $d_{TV}(X,Y) := \sup_A \abs{\pr(X\in A) - \pr(Y\in A)}$.
\end{lem}

Our theorem below shows that, in the ferromagnetic regime, the SKFI free energy is indistinguishable at fluctuation scale from the SK free energy in the \emph{effective external field} $\gamma\mu+h$: the quadratic term $\frac{\gamma N}{2}m^2$ linearizes around the concentrated magnetization. Note that in the ferromagnetic regime the effective field is automatically non-degenerate: if $\gamma\mu+h=0$ with $\mu\in\Omega$, $\mu\neq0$ (so that in particular $\gamma>0$), then the variational functional at $\mu$ equals $F^{SK}(\mvgb,0) - \gamma\mu^2/2 < F^{SK}(\mvgb,0) \leq F(\mvgb, \gamma, h)$, contradicting the maximality of $\mu$. Hence $\gamma\mu+h \ne 0$ and Lemma~\ref{lem:SK_dist_limit} applies with $h$ replaced by $\gamma\mu+h$.

\subsection{Main result}\label{subsec:thmB}

\begin{thm}\label{thm:ferro}
Suppose $\Omega(\mvgb,\gamma,h) = \{\mu\}$ for some $\mu\ne0$, or $\Omega(\mvgb,\gamma,h) = \{-\mu,\mu\}$ for some $\mu>0$. Then
\begin{align*}
    \lim_{N\to\infty} N\cdot\var\Big( F_N(\mvgb,\gamma,h) - F_N^{SK}(\mvgb,\gamma\mu+h) \Big) = 0.
\end{align*}
\end{thm}

\begin{cor}\label{cor:ferro}
    Under the hypotheses of Theorem~\ref{thm:ferro}, with $\nu = \nu(\mvgb,\gamma\mu+h)$ defined in~\eqref{eq:nu},
    \begin{align*}
        \sqrt N\,\big( F_N(\mvgb,\gamma,h) - \bar F_N(\mvgb,\gamma,h) \big)
        \overset{d}{\longrightarrow} \cN(0,\nu).
    \end{align*}
\end{cor}

Indeed, $\sqrt N\big(F_N - \bar F_N\big)$ differs from $\sqrt N\big(F_N^{SK}(\mvgb,\gamma\mu+h) - \bar F_N^{SK}(\mvgb,\gamma\mu+h)\big)$ by a centered term that vanishes in $L^2$ by Theorem~\ref{thm:ferro}, and the latter converges to $\cN(0,\nu)$ by Lemma~\ref{lem:SK_dist_limit}, applicable since $\gamma\mu+h\ne0$. In particular, Corollary~\ref{cor:ferro} with $\beta_2=\beta/\sqrt2$, $\beta_p=0$ ($p\neq2$), $h=0$ yields Theorem~\ref{thm:A}(iii), by Remark~\ref{rem:convention}.

Theorem~\ref{thm:ferro} is proved in Section~\ref{sec:proof-ferro}. The proof combines three ingredients: a coupled Gaussian integration-by-parts identity expressing $N\var(F_N - F_N^{SK})$ as a $t$-integral of differences of overlap functionals (Lemma~\ref{lem:var}); the exponential concentration of the magnetization on $\Omega$~\cite{Chen14} (Lemma~\ref{lem:magnetization}); and the disorder chaos estimate of Chen~\cite{Chen13}, which pins the coupled overlaps at the deterministic value $u_t$ (Lemmas~\ref{lem:chaos} and~\ref{lem:overlap_concentration}).

%%%%%%%%%%%%%%%%%%%%%%%%%%%%%%%%
\section{Proof of Proposition~\ref{prop:magnetization}}\label{sec:proof-mag}

\subsection{Proof of Proposition~\ref{prop:magnetization}(i)}
\begin{proof}
    By Jensen's inequality, we have
    \begin{align*}
        \bar F^{SK}_N (\beta, \gamma \mu) - \bar F^{SK}_N (\beta, 0)
        &= \frac{1}{N} \E \ln \Big\langle \exp \Big( \gamma \mu \sum_{i \leq N} \sigma_i \Big)\Big\rangle^{SK}_{(\beta, 0)} \\
        &\leq \frac{1}{N} \ln \E \Big\langle \exp \Big( \gamma \mu \sum_{i \leq N} \sigma_i \Big)\Big\rangle^{SK}_{(\beta, 0)}
        = \ln \cosh(\gamma \mu).
    \end{align*}
    In the last equality, we used the fact that, by symmetry, the disorder-averaged Gibbs measure $\E G^{SK}_{N}(\cdot)$ is uniform over $\Sigma_N$ in the absence of an external field.
    Consequently, we can bound the variational function as follows:
    \begin{align*}
        \bar F^{SK}_N (\beta, \gamma \mu) - \frac{\gamma \mu^2}{2} - \bar F^{SK}_N (\beta, 0)
        &\leq \left( \ln \cosh(\gamma \mu) -\frac{\gamma^2 \mu^2}{2} \right) + \frac{(\gamma - 1)\gamma \mu^2}{2}.
    \end{align*}
    For any $\mu \neq 0$, the elementary bound $\cosh z < \exp\left( \tfrac{z^2}{2} \right)$ for $z \neq 0$ implies that the first term on the right-hand side is strictly negative. For $0 < \gamma \leq 1$, the second term is non-positive. Since this strictly negative upper bound is independent of $N$, passing to the thermodynamic limit $N \to \infty$ yields
    \begin{align*}
        F^{SK} (\beta, \gamma \mu) - \frac{\gamma \mu^2}{2} < F^{SK} (\beta, 0).
    \end{align*}
    Therefore, the supremum in the variational formula is uniquely achieved at $\mu = 0$, and the claim $ \Omega(\beta, \gamma) = \{0\} $ follows.
\end{proof}

\subsection{Proof of Proposition~\ref{prop:magnetization}(ii)}
\begin{proof}
    By~\cite[Theorem 1.1]{BY22}, for all $(\beta, h)$ satisfying
    \begin{align}
        \beta^2  \E \sech^2 (\beta \sqrt{q} z + h) \leq 1,
        \label{eq:RS_regime}
    \end{align}
    where $z$ denotes a standard Gaussian random variable and $q = q(\beta,h) \in [0,1]$ is the unique solution to
    \begin{align*}
        q = \E \tanh^2 (\beta \sqrt{q} z + h),
    \end{align*}
    the free energy admits the representation
    \begin{align*}
        F^{SK}(\beta,h) = \ln 2 + \E \ln \cosh(\beta \sqrt{q} z + h) + \frac{\beta^2}{4} (1-q)^2.
    \end{align*}
    Fix $\beta \in [0,1]$. Then, for any $h$, the pair $(\beta, h)$ satisfies~\eqref{eq:RS_regime}.

    Applying Taylor's theorem, there exists some $c \in [0,h]$ such that
    \begin{align*}
        q = \E \tanh^2 (\beta \sqrt{q} z + h) = \E \tanh^2 (\beta \sqrt{q} z)
          + \frac{1}{2} \partial_h^2 \E \tanh^2 (\beta \sqrt{q} z + h)\big|_{h=c} \cdot h^2.
    \end{align*}
    Since $ \frac{d^2}{d\theta^2} \tanh^2 \theta$ is uniformly bounded and $ \tanh^2 (\beta \sqrt{q} z ) \leq \beta^2 q z^2$, there exists a universal constant $C>0$ such that
    \begin{align*}
        0 \leq q - \E \tanh^2 (\beta \sqrt{q} z) \leq C h^2.
    \end{align*}
    This implies
    \begin{align}
        q - \E \tanh^2 (\beta \sqrt{q} z) \to 0 \quad \text{as } h \to 0.
        \label{eq:q-Eth->0}
    \end{align}
    Moreover, using the bounds $ \abs{\tanh \theta} \leq \theta $ and $ \sech^2 \theta < 1$, we obtain
    \begin{align*}
        \frac{d}{dq} \left( q - \E \tanh^2 (\beta \sqrt{q} z) \right)
        = 1 -  \frac{\beta}{\sqrt{q}}\E \big[ \tanh (\beta \sqrt{q} z) \sech^2 (\beta \sqrt{q} z) z \big] 
        > 1 - \beta^2.
    \end{align*}
    This implies that the function $q \mapsto q - \E \tanh^2 (\beta \sqrt{q} z)$ is strictly increasing for $q \in \dR^+$. Combining this with~\eqref{eq:q-Eth->0}, we deduce
    \begin{align}
        \lim_{h \to 0+} q = 0.
        \label{eq:q->0}
    \end{align}

    By the smart path method (Guerra's replica symmetric bound; see~\cite{Gue03} and~\cite[Theorem 1.3.7]{Talagrand10}), for any $r \in [0,1]$ one has
    \begin{align}
        F^{SK}(\beta,h) \leq \ln 2 + \E \ln \cosh(\beta \sqrt{r} z + h) + \frac{\beta^2}{4} (1-r)^2.
        \label{eq:Guerra_RS_bound}
    \end{align}
    Hence, the left derivative satisfies
    \begin{align*}
        \lim_{\varepsilon \to 0+} \frac{F^{SK}(\beta,h) - F^{SK}(\beta,h-\varepsilon)}{\varepsilon}
        &\geq \lim_{\varepsilon \to 0+} \frac{1}{\varepsilon}
        \Big( \E \ln \cosh(\beta \sqrt{q} z + h) - \E \ln \cosh(\beta \sqrt{q} z + h - \varepsilon) \Big) \\
        &= \E \tanh(\beta \sqrt{q} z + h),
    \end{align*}
    and similarly, the right derivative satisfies
    \begin{align*}
        \lim_{\varepsilon \to 0+} \frac{F^{SK}(\beta,h+\varepsilon) - F^{SK}(\beta,h)}{\varepsilon} \leq \E \tanh (\beta \sqrt{q} z + h).
    \end{align*}
    Since $F^{SK}(\beta,h)$ is convex in $h$, it follows that it is differentiable and
    \begin{align*}
        \partial_h F^{SK}(\beta,h) = \E \tanh (\beta \sqrt{q} z + h).
    \end{align*}

    Next, note that $\tanh''(\theta) = -2 \sech^2(\theta) \tanh(\theta)$ is uniformly bounded. Hence, by Taylor's theorem,
    \begin{align*}
        \abs{ \E \tanh (\beta \sqrt{q} z + h) - \E \sech^2 (\beta \sqrt{q} z) h } \leq C h^2.
    \end{align*}
    Combining this with~\eqref{eq:q->0}, we obtain
    \begin{align*}
        \abs{ \frac{\partial_h F^{SK}(\beta,h)}{h} - 1 }
        &\leq \frac{1}{h} \abs{ \partial_h F^{SK}(\beta,h) - \E \sech^2 (\beta \sqrt{q} z) h }
        + \abs{ \E \sech^2 (\beta \sqrt{q} z) - 1 } \\
        &\leq C h + \abs{ \E \tanh^2 (\beta \sqrt{q} z)}
        \xrightarrow[h \to 0+]{} 0.
    \end{align*}

    Therefore, for any $\gamma > 1$, there exists $\delta > 0$ such that for all $h \in (0,\delta)$,
    \begin{align*}
        \frac{\partial_h F^{SK}(\beta,h)}{h} > \frac{1}{\gamma}.
    \end{align*}
    Integrating this inequality yields
    \begin{align*}
        F^{SK}(\beta, h) - \frac{h^2}{2\gamma} > F^{SK}(\beta,0), \qquad h \in (0,\delta).
    \end{align*}
    By setting $h = \gamma \mu$ for a sufficiently small $\mu > 0$, the variational function satisfies $F^{SK}(\beta, \gamma \mu) - \frac{\gamma \mu^2}{2} > F^{SK}(\beta, 0)$. This strictly positive difference guarantees that $\mu=0$ cannot be the global maximizer, which concludes $0 \notin \Omega(\beta, \gamma)$.
\end{proof}

\subsection{Proof of Proposition~\ref{prop:magnetization}(iii)}
\begin{proof}
    By Jensen's inequality, for any $h\in\dR$,
    \begin{align*}
        \bar F^{SK}_N (\beta, h) - \bar F^{SK}_N (0, h)
        = \frac{1}{N} \E \ln \Big\la \exp \big(H_N^{SK}(\mvgs) \big)\Big\ra^{SK}_{(0, h)}
        \geq \frac{1}{N} \E \Big\la H_N^{SK}(\mvgs) \Big\ra^{SK}_{(0, h)} = 0,
    \end{align*}
    while $\bar F^{SK}_N (0, h) = \ln 2 + \ln\cosh h$ exactly, since the spins are independent at $\beta = 0$. Passing to the thermodynamic limit and evaluating at $h=\gamma\theta$, we obtain the master inequality
    \begin{align}
        F^{SK}(\beta, \gamma \theta) - \frac{\gamma \theta^2}{2} - F^{SK}(\beta, 0)
        \;\geq\; \ln \cosh(\gamma \theta) - \frac{\gamma \theta^2}{2} - \big( F^{SK}(\beta,0) - \ln 2 \big),
        \label{eq:mag_master}
    \end{align}
    so that any upper bound on $F^{SK}(\beta,0) - \ln2$ produces a portion of the ferromagnetic phase. Here we use Guerra's replica symmetric bound~\eqref{eq:Guerra_RS_bound} at $h=0$: for every $r\in[0,1]$,
    \begin{align*}
        F^{SK}(\beta,0) - \ln 2 \;\leq\; \E \ln \cosh\big(\beta \sqrt{r}\, z \big) + \frac{\beta^2}{4}\, (1-r)^2 .
    \end{align*}
    For $\beta>1$, the map $r \mapsto \E \ln \cosh(\beta \sqrt{r}\, z) + \frac{\beta^2}{4}(1-r)^2$ has a unique minimizer $q \in (0,1)$, characterized by the stationarity condition $q = \E\tanh^2(\beta\sqrt q\, z)$; for $\beta\le1$ the minimum is attained at $q=0$, where the bound reduces to the annealed value $\beta^2/4$. Combining the bound at $r=q$ with~\eqref{eq:mag_master}, the hypothesis of (iii) gives
    \begin{align*}
        F^{SK}(\beta,\gamma\theta) - \frac{\gamma\theta^2}{2} - F^{SK}(\beta,0) > 0.
    \end{align*}
    Since $\theta = \theta(\gamma)$ is admissible ($\theta \in (0,1)$ solves the mean-field equation $\theta = \tanh(\gamma\theta)$ and maximizes $\theta \mapsto \ln\cosh(\gamma\theta) - \gamma\theta^2/2$ on $[0,1]$), choosing $\mu = \theta$ yields a strictly larger variational value than $\mu = 0$; hence $0 \notin \Omega(\beta, \gamma)$.

    It remains to verify the explicit sufficient condition. Taking $r=0$ and $r=1$ in the replica symmetric bound,
    \begin{align*}
        \inf_{r \in [0,1]} \left( \E \ln \cosh\big(\beta \sqrt{r}\, z \big) + \frac{\beta^2}{4}\, (1-r)^2 \right)
        \;\leq\; \min \left\{\frac{\beta^2}{4},\ \E \ln \cosh (\beta z)  \right\}.
    \end{align*}
    Using the identity $\ln \cosh x = \abs{x} - \ln 2 + \ln \big(1 + e^{-2 \abs{x}}\big)$, the bound $\ln(1+u) \leq u$, the value $\E\abs z = \sqrt{2/\pi}$, we obtain
    \begin{align*}
        \E \ln \cosh (\beta z) \;\leq\; \sqrt{\frac{2}{\pi}} \left( \beta + \frac{1}{2\beta} \right) - \ln 2.
    \end{align*}
    On the other side, since $\theta(\gamma)$ maximizes $\theta \mapsto \ln\cosh(\gamma\theta) - \gamma\theta^2/2$ on $[0,1]$, we may lower bound the left-hand side of the hypothesis of (iii) by its value at $\theta = 1$; together with $\ln \cosh \gamma \geq \gamma - \ln 2$, this gives
    \begin{align*}
        \ln \cosh(\gamma\theta) - \frac{\gamma \theta^2}{2}
        \;\geq\; \ln\cosh\gamma - \frac\gamma2
        \;\geq\; \frac{\gamma}{2} - \ln 2 .
    \end{align*}
    Hence the hypothesis of (iii) holds as soon as
    \begin{align*}
        \frac{\gamma}{2} - \ln 2 \;>\; \min \left\{\frac{\beta^2}{4},\ \sqrt{\frac{2}{\pi}} \left( \beta + \frac{1}{2\beta} \right) - \ln 2 \right\},
    \end{align*}
    which is precisely the condition $\gamma > \min\big\{\beta^2/2 + 2\ln2,\ \sqrt{2/\pi}\,(2\beta + 1/\beta)\big\}$. This completes the proof.
\end{proof}

\subsection{Proof of Proposition~\ref{prop:magnetization}(iv)}
The starting point is the following exact identity for the internal energy on the line $\gamma=\beta^2$. This is a special case of the Nishimori identities~\cite{Nis81,Nis01} and follows from the standard Nishimori gauge argument.
\begin{lem}\label{lem:Nishimori}
    For inverse temperature parameters $\beta$ and $\gamma = \beta^2$, we have
    \[\E \la H_N(\mvgs) \ra_{(\beta, \beta^2)} = \frac{N}{2} \beta^2.\]
\end{lem}
\begin{proof}
    Recall the definition of the Hamiltonian at inverse temperature parameters $\beta, \gamma$ is
    \begin{align*}
        H_N(\mvgs) = \sum_{i<j} \left( \frac{\beta}{\sqrt{N}}g_{ij} + \frac{\gamma}{N} \right) \sigma_i \sigma_j + \frac{\gamma}{2}.
    \end{align*}
    Using the substitutions $\omega_{ij} = \frac{\beta}{\sqrt{N}}g_{ij} + \frac{\gamma}{N}$, we obtain
    \begin{align*}
        H_N(\mvgs) = \sum_{i<j} \omega_{ij} \sigma_i \sigma_j + \frac{\gamma}{2},
    \end{align*}
    where $\omega_{ij}$ are i.i.d. random variables following $\cN(\frac{\gamma}{N},\frac{\beta^2}{N})$. Thus, taking the expectation over the disorder $\omega_{ij}$, we have
    \begin{align*}
        &\E \la H_N(\mvgs) \ra -\frac{\gamma}{2} \\
        &=\int_{\dR^{\binom{N}{2}}} \frac{ \sum_{\mvgs} \left( \sum_{i<j} \omega_{ij} \sigma_i \sigma_j \right) \exp\left(\sum_{i<j} \omega_{ij} \sigma_i \sigma_j\right)}{\sum_{\mvgt} \exp \left(\sum_{i<j} \omega_{ij} \tau_i \tau_j\right)} \prod_{i<j} \left( \sqrt{\frac{N}{2\pi \beta^2}} \exp \left( -\frac{N}{2 \beta^2}\left(\omega_{ij}- \frac{\gamma}{N} \right)^2 \right) d\omega_{ij} \right).
    \end{align*}
    For any fixed spin configuration $\varepsilon \in \Sigma_N$, if we apply the gauge transformation $\omega_{ij} \mapsto \varepsilon_i \varepsilon_j \omega_{ij}$ for $1 \leq i<j \leq N$, the term
    \begin{align*}
        \frac{ \sum_{\mvgs} \left( \sum_{i<j} \omega_{ij} \sigma_i \sigma_j \right) \exp\left(\sum_{i<j} \omega_{ij} \sigma_i \sigma_j\right)}{\sum_{\mvgt} \exp \left(\sum_{i<j} \omega_{ij} \tau_i \tau_j\right)}
    \end{align*}
    remains invariant. This is because we can reparameterize the dummy variables in the sum by $\sigma_i \mapsto \varepsilon_i \sigma_i$ and $\tau_i \mapsto \varepsilon_i \tau_i$, which leaves the sum unchanged. Since this holds for all $2^N$ configurations of $\varepsilon$, we can average over them to obtain
    \begin{align*}
        \E \la H_N(\mvgs) \ra -\frac{\gamma}{2}
        &=\frac{1}{2^N} \int_{\dR^{\binom{N}{2}}} \frac{ \sum_{\mvgs} \left( \sum_{i<j} \omega_{ij} \sigma_i \sigma_j \right) \exp\left(\sum_{i<j} \omega_{ij} \sigma_i \sigma_j\right)}{\sum_{\mvgt} \exp \left(\sum_{i<j} \omega_{ij} \tau_i \tau_j\right)} \\
        &\qquad \times \sum_{\varepsilon \in \Sigma_N}\prod_{i<j} \left( \sqrt{\frac{N}{2\pi \beta^2}} \exp \left( -\frac{N}{2 \beta^2}\left(\varepsilon_i \varepsilon_j \omega_{ij}- \frac{\gamma}{N} \right)^2 \right) d\omega_{ij} \right).
    \end{align*}
    Expanding the square inside the Gaussian density, we get
    \begin{align*}
        \sum_{\varepsilon \in \Sigma_N}\prod_{i<j} & \left( \exp \left( -\frac{N}{2 \beta^2}\left(\varepsilon_i \varepsilon_j \omega_{ij}- \frac{\gamma}{N} \right)^2 \right) \right) \\
        &= \exp \left( \sum_{i<j} \left( -\frac{N}{2 \beta^2} \omega_{ij}^2 - \frac{\gamma^2}{2 \beta^2 N} \right) \right) \sum_{\varepsilon \in \Sigma_N} \exp \left( \frac{\gamma}{\beta^2}\sum_{i<j} \omega_{ij} \varepsilon_i \varepsilon_j \right).
    \end{align*}
    Crucially, if we enforce the Nishimori condition $\gamma = \beta^2$, the sum over $\varepsilon$ exactly cancels the partition function in the denominator. Thus, it follows that
    \begin{align*}
        &\E \la H_N(\mvgs) \ra -\frac{\gamma}{2} \\
        &=\frac{1}{2^N} \sum_{\mvgs} \int_{\dR^{\binom{N}{2}}} \left( \sum_{i<j} \omega_{ij} \sigma_i \sigma_j \right) \exp\left(\sum_{i<j} \omega_{ij} \sigma_i \sigma_j\right) \prod_{i<j} \left( \sqrt{\frac{N}{2\pi \beta^2}} \exp \left( -\frac{N}{2 \beta^2} \omega_{ij}^2 - \frac{\gamma^2}{2 \beta^2 N} \right) d\omega_{ij} \right).
    \end{align*}
    By completing the square in the exponent and substituting $z_{ij} = \frac{\sqrt{N}}{\beta} \omega_{ij} - \frac{\gamma}{\beta \sqrt{N}} \sigma_i \sigma_j$, the integral simplifies to a standard Gaussian expected value:
    \begin{align*}
        \E \la H_N(\mvgs) \ra
        &= \frac{\gamma}{2} + \frac{1}{2^N} \sum_{\mvgs} \int_{\dR^{\binom{N}{2}}} \left( \sum_{i<j} \left( \frac{\beta}{\sqrt{N}} z_{ij} + \frac{\gamma}{N} \sigma_i \sigma_j  \right) \sigma_i \sigma_j \right) \prod_{i<j} \left( \frac{1}{\sqrt{2\pi}} \exp \left( -\frac{z_{ij}^2}{2} \right) dz_{ij} \right).
    \end{align*}
    Since the integral of the linear term $z_{ij}$ vanishes under the standard normal distribution, and $(\sigma_i \sigma_j)^2 = 1$, we evaluate the sum to get
    \begin{align*}
        \E \la H_N(\mvgs) \ra = \frac{\gamma}{2} + \sum_{i<j} \frac{\gamma}{N} = \frac{\gamma}{2} + \frac{N(N-1)}{2} \frac{\gamma}{N} = \frac{N}{2} \gamma.
    \end{align*}
    Since $\gamma = \beta^2$, this yields $\frac{N}{2} \beta^2$. This completes the proof.
\end{proof}

\begin{proof}[Proof of Proposition~\ref{prop:magnetization}(iv)]
    By definition, for fixed $J \geq 0$, we have
    \begin{align*}
        \partial_{\beta} \bar F_N(\beta, \beta J) = \frac{1}{\beta N} \E \la H_N(\mvgs) \ra_{(\beta, \beta J)}.
    \end{align*}
    Since the map $\gamma \mapsto \bar F_N(\beta, \gamma)$ is non-decreasing on $\gamma \geq 0$, the total derivative of the map $\beta \mapsto \bar F_N(\beta, \beta^2)$ satisfies
    \begin{align*}
        \partial_{\beta} \bar F_N(\beta, \beta^2)
        &= \lim_{\varepsilon \to 0} \frac{1}{\varepsilon} \left( \bar F_N(\beta + \varepsilon, (\beta + \varepsilon)^2) - \bar F_N(\beta, \beta^2) \right) \\
        &\geq \lim_{\varepsilon \to 0} \frac{1}{\varepsilon} \left( \bar F_N(\beta + \varepsilon, (\beta + \varepsilon)\beta ) - \bar F_N(\beta, \beta^2) \right)\\
        &= \frac{1}{\beta N} \E \la H_N(\mvgs) \ra_{(\beta, \beta^2)}.
    \end{align*}
    For $\beta > 1$, integrating this inequality and applying Lemma~\ref{lem:Nishimori}, we obtain
    \begin{align*}
        \bar F_N(\beta, \beta^2 ) - \bar F_N(1, 1)
        = \int_1^{\beta} \partial_x \bar F_N(x, x^2)\, dx
        \geq \int_1^{\beta} \frac{x}{2}\, dx = \frac{\beta^2 - 1}{4}.
    \end{align*}
    Taking the thermodynamic limit $N \to \infty$ yields
    \begin{align*}
        F(\beta, \beta^2 ) - F(1, 1) \geq \frac{\beta^2 - 1}{4}.
    \end{align*}
    By Proposition~\ref{prop:magnetization}(i), we know that $F(1,1) = F^{SK}(1,0)$. Furthermore, it is a well-known result that $F^{SK}(\beta,0) = \ln 2 + \beta^2 /4$ for $\beta < 1$ (see, e.g.,~\cite{ALR87}). Using the continuity of the map $\beta \mapsto F^{SK}(\beta,0)$, we deduce that $F^{SK}(1,0) = \ln 2 + 1/4$. Substituting this into the inequality above gives
    \[F(\beta, \beta^2) \geq \ln 2 + \frac{\beta^2}{4}.\]

    On the other hand, for the zero-field free energy at low temperature ($\beta > 1$), it is strictly bounded as
    \[F^{SK}(\beta, 0 ) < \ln2 + \frac{\beta^2}{4}\]
    (see~\cite{Ton02} and~\cite[Theorem 13.3.1]{Talagrand11}).
    Comparing the two bounds implies $F(\beta, \beta^2) > F^{SK}(\beta, 0)$. This strict inequality indicates that $0 \not\in \Omega(\beta, \beta^2) $ for $\beta > 1$.

    Finally, since the map $\gamma \mapsto F(\beta, \gamma)$ is non-decreasing on $\gamma \geq 0$, for $\beta>1$ and $\gamma \geq \beta^2$ we obtain
    \[ F(\beta, \gamma) \geq F(\beta, \beta^2) > F^{SK}(\beta,0).\]
    Therefore, $0 \notin \Omega(\beta, \gamma)$ for $\beta>1$ and $\gamma \geq \beta^2$. This completes the proof.
\end{proof}

\subsection{Proof of Proposition~\ref{prop:magnetization}(v)}
\begin{proof}
    By the master inequality~\eqref{eq:mag_master} from the proof of (iii), it suffices to bound $F^{SK}(\beta,0) - \ln 2$ from above; here we replace the replica symmetric bound used there by a ground-state bound, which is sharper for very large $\beta$. Write $H_N^{(1)}$ for the SK Hamiltonian at $\beta=1$, so that $H_N^{SK} = \beta H_N^{(1)}$ in distribution, and set
    \begin{align*}
        GS_N := \frac1N\, \E\max_{\mvgs\in\Sigma_N} H_N^{(1)}(\mvgs).
    \end{align*}
    Bounding the partition function by $2^N$ times its largest term,
    \begin{align}
        \bar F_N^{SK}(\beta,0) \;\le\; \ln 2 + \frac{\beta}{N}\,\E\max_{\mvgs\in\Sigma_N} H_N^{(1)}(\mvgs) \;=\; \ln2 + \beta\, GS_N .
        \label{eq:GS-bound}
    \end{align}
    The sequence $(N\, GS_N)_{N\ge1}$ is superadditive --- for instance, by dividing the Guerra--Toninelli interpolation~\cite{GT02} by $\beta$ and letting $\beta\to\infty$ --- so by Fekete's lemma the limit $E_0 = \lim_N GS_N$ exists and $GS_N \le E_0$ for every $N$. Hence, taking $N\to\infty$ in~\eqref{eq:GS-bound},
    \begin{align*}
        F^{SK}(\beta,0) - \ln2 \;\le\; \beta E_0 .
    \end{align*}
    Combining this bound with the master inequality~\eqref{eq:mag_master} gives
    \begin{align*}
        F^{SK}(\beta,\gamma\theta) - \frac{\gamma\theta^2}{2} - F^{SK}(\beta,0)
        \;\ge\; \ln\cosh(\gamma\theta) - \frac{\gamma\theta^2}{2} - \beta E_0 \;>\; 0
    \end{align*}
    under the hypothesis of (v), so the variational maximum is not attained at $\mu=0$ and $0\notin\Omega(\beta,\gamma)$.

    It remains to prove $E_0 \le \sqrt{\ln2}$. For each fixed $\mvgs\in\Sigma_N$, the random variable $H_N^{(1)}(\mvgs)$ is a centered Gaussian with variance
    \begin{align*}
        \var\big( H_N^{(1)}(\mvgs) \big) = \frac1N\binom{N}{2} = \frac{N-1}{2}.
    \end{align*}
    By the standard Gaussian maximal inequality, the expected maximum of $2^N$ centered Gaussian random variables with common variance $s^2$ is at most $s\sqrt{2\ln(2^N)}$, whence
    \begin{align*}
        \E\max_{\mvgs\in\Sigma_N} H_N^{(1)}(\mvgs)
        \;\le\; \sqrt{\frac{N-1}{2}}\cdot\sqrt{2N\ln2}
        \;=\; \sqrt{N(N-1)\ln2}
        \;\le\; N\sqrt{\ln2}.
    \end{align*}
    Thus $GS_N \le \sqrt{\ln2}$ for every $N$, and letting $N\to\infty$ gives $E_0 \le \sqrt{\ln2}$. 
\end{proof}

%%%%%%%%%%%%%%%%%%%%%%%%%%%%%%%%
\section{The paramagnetic and spin glass regimes}\label{sec:proof-para}

The heart of Theorem~\ref{thm:A}(i) is the following comparison of partition functions, which shows that at scale $O(1)$ the ferromagnetic interaction contributes only the deterministic factor $(1-\gamma_N)^{-1/2}$. Theorem~\ref{thm:A}(i) follows immediately: $\ln\big(\sqrt{1-\gamma_N}\,Z_N(\beta,\gamma_N)/Z_N^{SK}(\beta,0)\big) \to 0$ in probability, and adding the zero-field CLT~\eqref{eq:beta<1_dist_lim} yields the stated limit.

\begin{thm}\label{thm:ratio}
For $\beta< 1$ and a sequence of non-negative numbers $ (\gamma_N)_{N \geq 1}$ satisfying $\gamma_N < 1$ and $ \sqrt{N} (1 - \gamma_N) \to \infty$, we have
\begin{align*}
\sqrt{1-\gamma_N}\, \frac{Z_N(\beta, \gamma_N)}{Z_N^{SK}(\beta, 0)} \longrightarrow 1 \quad \text{in } L^1.
\end{align*}
\end{thm}

\begin{proof}[Proof of Theorem~\ref{thm:ratio}]
    By applying the Hubbard–Stratonovich transform, we obtain
    \begin{align*}
        Z_N(\beta, \gamma_N)
        &= \sum_{\mvgs \in \Sigma_N} \exp \left( H_N^{SK}(\mvgs) + \frac{\gamma_N N}{2} m(\mvgs)^2 \right) \\
        &= \frac{1}{\sqrt{2 \pi}} \int_{\dR} \sum_{\mvgs \in \Sigma_N} \exp \left( H_N^{SK}(\mvgs) + \sqrt{\gamma_N N} m(\mvgs) x - {x^2}/{2} \right) dx \\
        &= \frac{1}{\sqrt{2 \pi}} \int_\dR Z_N^{SK}(\beta, \sqrt{\gamma_N/N}x) \exp(-x^2/2) dx.
    \end{align*}
    From this representation, together with the bound $ \cosh \theta \leq \exp(\theta^2/2) $, it follows that
    \begin{align*}
        &\sqrt{1 - \gamma_N} \E \abs{\frac{Z_N ( \beta, \gamma_N)}{Z_N^{SK} ( \beta, 0)} - \frac{1}{\sqrt{2 \pi}}
        \int_\dR \cosh^N(\sqrt{\gamma_N/N}x) \exp(-x^2/2) dx}  \\
        &\leq \frac{\sqrt{1 - \gamma_N} }{\sqrt{2 \pi}} \int_{\dR}
        \E \abs{\frac{Z_N^{SK} ( \beta, \sqrt{\gamma_N / N} x)}{Z_N^{SK} ( \beta, 0) \cosh^N(\sqrt{\gamma_N / N}x)} -1 }
        \cosh^N(\sqrt{\gamma_N/N}x) \exp(-x^2/2) dx \\
        &\leq \frac{\sqrt{1 - \gamma_N} }{\sqrt{2 \pi}} \int_{\dR}
        \E \abs{\frac{Z_N^{SK} ( \beta, \sqrt{\gamma_N / N} x)}{Z_N^{SK} ( \beta, 0) \cosh^N(\sqrt{\gamma_N / N}x)} -1 } \exp(-(1 - \gamma_N) x^2/2) dx \\
        &= \frac{1}{\sqrt{2 \pi}} \int_{\dR}
        \E \abs{\frac{Z_N^{SK} \Big( \beta, \sqrt{ \frac{\gamma_N}{(1 - \gamma_N)N}} y \Big)}{Z_N^{SK} ( \beta, 0) \cosh^N\Big(\sqrt{ \frac{\gamma_N}{(1 - \gamma_N)N}} y \Big)} -1 } \exp(-y^2/2) dy.
    \end{align*}
    In the last equality, we applied the change of variables $y = \sqrt{1 - \gamma_N} x$.
    Since the disorder-averaged Gibbs measure $\E G^{SK}_N(\cdot)$ is uniform by symmetry, we have, for any $h \in \dR $,
    \begin{align*}
        \E \left[ \frac{Z_N^{SK} ( \beta, h)}{Z_N^{SK} ( \beta, 0)\cosh^N(h)}\right] = 1,
    \end{align*}
    and therefore
    \begin{align*}
        \E \abs{ \frac{Z_N^{SK} ( \beta, h)}{Z_N^{SK} ( \beta, 0) \cosh^N(h)} - 1 } \leq 2.
    \end{align*}
    Moreover, under the condition that $ N \left( \sqrt{ \frac{\gamma_N}{(1 - \gamma_N)N}} y \right)^4 \to 0 $,~\cite[Theorem 2.1]{DW23} asserts that
    \begin{align*}
        \frac{Z_N^{SK} \Big( \beta, \sqrt{ \frac{\gamma_N}{(1 - \gamma_N)N}} y \Big)}{Z_N^{SK} ( \beta, 0) \cosh^N\Big(\sqrt{ \frac{\gamma_N}{(1 - \gamma_N)N}} y \Big)} \to 1
        \quad \text{as } N \to \infty \text{ in } L^1.
    \end{align*}
    Combining this fact with the Dominated Convergence Theorem, we deduce
    \begin{align}
        \sqrt{1 - \gamma_N} \E \abs{\frac{Z_N ( \beta, \gamma_N)}{Z_N^{SK} ( \beta, 0)} - \frac{1}{\sqrt{2 \pi}}
        \int_\dR \cosh^N(\sqrt{\gamma_N/N}x) \exp(-x^2/2) dx}
        \to 0 \quad \text{as } N \to \infty.
        \label{eq:para_conv}
    \end{align}
    Substituting $y = \sqrt{1 - \gamma_N} x$ into the integral, we obtain
    \begin{align*}
        &\frac{\sqrt{1-\gamma_N}}{\sqrt{2 \pi}} \int_{\dR} \cosh^N (\sqrt{\gamma_N/N}x) \exp(-x^2/2) dx \\
        &= \frac{1}{\sqrt{2 \pi}} \int_{\dR} \cosh^N \left( \sqrt{ \frac{\gamma_N}{(1 - \gamma_N)N }} y \right) \exp \left(- \frac{\gamma_N}{2(1 - \gamma_N) } y^2 \right) \exp(-y^2/2) dy.
    \end{align*}
    Using Taylor expansion, we have
    \begin{align*}
        N \ln \cosh \left( \sqrt{ \frac{\gamma_N}{(1 - \gamma_N)N }} y \right) - \frac{\gamma_N}{2(1 - \gamma_N) } y^2
        = O \left( \frac{\gamma_N^2}{N(1 - \gamma_N)^2 } y^4 \right).
    \end{align*}
    Thus, under the condition that $ \sqrt{N} (1 - \gamma_N) \to \infty $,
    \begin{align*}
        \cosh^N \left( \sqrt{ \frac{\gamma_N}{(1 - \gamma_N)N }} y \right) \exp \left(- \frac{\gamma_N}{2(1 - \gamma_N) } y^2 \right) \to 1 \quad \text{as } N \to \infty.
    \end{align*}
    Furthermore, since $ \cosh(\theta) \le \exp(\theta^2/{2}) $, another application of the Dominated Convergence Theorem yields
    \begin{align*}
        \frac{\sqrt{1-\gamma_N}}{\sqrt{2 \pi}} \int_{\dR} \cosh^N (\sqrt{\gamma_N/N}x) \exp(-x^2/2) dx \to 1 \quad \text{as } N \to \infty.
    \end{align*}
    Combining this with~\eqref{eq:para_conv}, we conclude the proof.
\end{proof}

\begin{proof}[Proof of Proposition~\ref{prop:SG}]
    Since $\frac{\gamma N}{2}m(\mvgs)^2 \ge 0$, we have $Z_N(\beta,\gamma) \ge Z_N^{SK}(\beta,0)$ pointwise, which gives the first claim. For the tail bound, Markov's inequality yields, for $t\ge0$,
    \begin{align*}
        \pr\Big( \ln Z_N(\beta,\gamma) - \ln Z_N^{SK}(\beta,0) \ge t \Big)
        = \pr\Big( \frac{Z_N(\beta,\gamma)}{Z_N^{SK}(\beta,0)} \ge e^{t} \Big)
        \le e^{-t}\, \E\left[ \frac{Z_N(\beta,\gamma)}{Z_N^{SK}(\beta,0)} \right].
    \end{align*}
    By the symmetry of the zero-field Gibbs measure under disorder average,
    \begin{align*}
        \E \left[ \frac{Z_N(\beta,\gamma)}{Z_N^{SK}(\beta,0)} \right]
        = \frac{1}{2^N} \sum_{\mvgs \in \Sigma_N} \exp \left( \frac{\gamma}{2N} \Big( \sum_{i \leq N} \sigma_i \Big)^2 \right)
        \le \frac{1}{\sqrt{1-\gamma}},
    \end{align*}
    where the last inequality is standard (see, e.g., (A.24) in~\cite{Talagrand10}). This proves the tail bound. If $\nu_N \to \infty$, then applying the bound with $t = \varepsilon\sqrt{\nu_N}$ and using the almost sure lower bound gives
    \begin{align*}
        \frac{\ln Z_N(\beta,\gamma) - \ln Z_N^{SK}(\beta,0)}{\sqrt{\nu_N}} \overset{p}{\longrightarrow} 0.
    \end{align*}
    This completes the proof.
\end{proof}

%%%%%%%%%%%%%%%%%%%%%%%%%%%%%%%%
\section{The critical line: proof of Theorem~\ref{thm:A}(ii)}\label{sec:proof-crit}

In this section we prove Theorem~\ref{thm:A}(ii) in the following stronger, universal form: the Gaussian disorder may be replaced by an arbitrary symmetric disorder with unit variance and finite moments of all orders. Let $J$ be a symmetric random variable with $\E J^2 = 1$ and $\E \abs{J}^k < \infty$ for every $k\in\bN$, let $(J_{ij})_{i<j}$ be i.i.d.~copies of $J$, and let $Z_N^{(J)}(\beta,\gamma)$ denote the partition function of the SKFI model~\eqref{eq:H_SKFI_2spn} with the Gaussian couplings $(g_{ij})$ replaced by $(J_{ij})$; write $F_N^{(J)} := N^{-1}\ln Z_N^{(J)}$.

\begin{thm}\label{thm:g=1_general}
    Let $\beta<1$ and $\gamma_N = 1 + \theta N^{-1/2}$ for some $\theta\in\dR$. Then
    \begin{align*}
        N F_N^{(J)}(\beta,\gamma_N)&- N\ln2 - \frac{N-1}{4}\beta^2 - \frac14\ln N \\
        &\overset{d}{\longrightarrow}
        \zeta_1 + \ln \int_{\dR} \exp\Big( \Big(\zeta_2 + \frac{\theta}{2}\Big)y^2 - \Big(\frac1{12} + \frac{v_2^2}{2}\Big)y^4 \Big)\,\frac{dy}{\sqrt{2\pi}},
    \end{align*}
    where
    \begin{align*}
        v_1 = \Big( -\frac12\ln(1-\beta^2) - \frac{\beta^2}{2} - \frac{\beta^4}{4} \Big)^{1/2},
        \qquad
        v_2 = \Big( \frac{\beta^2}{2(1-\beta^2)} \Big)^{1/2},
    \end{align*}
    and $\zeta_1 \sim \cN\big( -\frac{\beta^4}{24}\,\E J^4 - \frac{v_1^2}{2},\; \frac{\beta^4}{8}(\E J^4 - 1) + v_1^2 \big)$ and $\zeta_2 \sim \cN(0, v_2^2)$ are independent.
\end{thm}

Theorem~\ref{thm:A}(ii) is the Gaussian case of Theorem~\ref{thm:g=1_general}: for $J \sim \cN(0,1)$ we have $\E J^4 = 3$, so that the quartic coefficient is unchanged ($v_2 = \alpha_2$), while the mean and variance of $\zeta_1$ reduce to
\begin{align*}
    -\frac{\beta^4}{8} - \frac{v_1^2}{2} = -\frac{\alpha_1^2}{2},
    \qquad
    \frac{\beta^4}{4} + v_1^2 = \alpha_1^2,
\end{align*}
since $\alpha_1^2 = v_1^2 + \frac{\beta^4}{4}$.

\subsection{Setting and outline of the proof}\label{subsec:g=1_setting}
Throughout this section we fix $\beta<1$, $\theta\in\dR$, and $\gamma_N = 1+\theta N^{-1/2}$, and we suppress the superscript $(J)$. Using the Hubbard--Stratonovich transform and the identity $\exp(\sigma\alpha) = \cosh\alpha\,(1+\sigma\tanh\alpha)$ for $\sigma\in\{-1,+1\}$, we can rewrite the partition function as
\begin{align}
    &Z_N (\beta, \gamma_N)\notag\\
    &= \prod_{i<j} \cosh \left(\frac{\beta J_{ij}}{\sqrt{N}} \right) \sum_{\mvgs} \prod_{i<j} \left( 1 + \sigma_i \sigma_j \tanh \left( \frac{\beta J_{ij}}{\sqrt{N}}\right)\right)
    \exp \left( \frac{\gamma_N}{2N} \Big( \sum_i \sigma_i \Big)^2 \right) \nonumber\\
    &= 2^N \prod_{i<j} \cosh \left(\frac{\beta J_{ij}}{\sqrt{N}} \right) \int_{\dR} \hat{Z}_N \left( \beta, x \sqrt{ \tfrac{\gamma_N}{N}} \right) \exp \left( N \ln \cosh \left( x \sqrt{ \tfrac{\gamma_N}{N}} \right) - \frac{x^2}{2} \right) \frac{dx}{\sqrt{2 \pi}} \nonumber\\
    &= \frac{2^N N^{1/4}}{\sqrt{2 \pi \gamma_N}} \prod_{i<j} \cosh \left(\frac{\beta J_{ij}}{\sqrt{N}} \right) \int_{\dR} \hat{Z}_N \left( \beta, y N^{-1/4} \right) \exp \left( N \ln \cosh \left( y N^{-1/4} \right) - \frac{\sqrt{N}}{2\gamma_N}\, y^2 \right) dy, \label{eq:g=1_Z_N}
\end{align}
where
\begin{align*}
    \hat{Z}_N (\beta, h) := \E_{\mvgs} \prod_{i<j} \left( 1 + \sigma_i \sigma_j \tanh \left( \frac{\beta J_{ij}}{\sqrt{N}}\right)\right) \prod_k \left(1 + \sigma_k \tanh h \right),
\end{align*}
and $\E_{\mvgs}$ denotes the expectation over $\mvgs \in \Sigma_N$ with respect to the uniform distribution. Note that $\hat Z_N > 0$ almost surely and $\E \hat Z_N = 1$.

By Taylor expansion, the exponent in~\eqref{eq:g=1_Z_N} satisfies
\begin{align*}
    N \ln \cosh \left( yN^{-1/4} \right) - \frac{\sqrt{N}y^2}{2\gamma_N}
    &= N \left( \frac{y^2}{2\sqrt{N}} - \frac{y^4}{12N} + \cO(y^6 N^{-3/2}) \right) - \frac{\sqrt{N}y^2}{2\gamma_N} \\
    &= \frac{\theta}{2 \gamma_N}\, y^2 - \frac{1}{12}\, y^4 + \cO(y^6 N^{-1/2}).
\end{align*}
Accordingly, we define
\begin{align*}
    A_N(y):= N \ln \cosh \left( yN^{-1/4} \right) - \frac{\sqrt{N}y^2}{2\gamma_N},
    \qquad
    \rho_{\theta} (y) := \exp \left( \frac{\theta}{2}\, y^2 - \frac{1}{12}\, y^4 \right).
\end{align*}
The first lemma justifies replacing $\exp(A_N)$ by the limiting density $\rho_\theta$.

\begin{lem}\label{lem:g=1_rho}
    For $\beta > 0$ and $\gamma_N = 1 + \theta N^{-1/2}$ with $\theta \in \dR$,
    \begin{align*}
        \int_{\dR} \hat{Z}_{N} \left( \beta, y N^{-1/4} \right) \left( \exp \left( A_N (y) \right)  - \rho_{\theta} (y) \right) dy
        \longrightarrow 0 \quad \text{in } L^1 \text{ as } N \to \infty.
    \end{align*}
\end{lem}

Since $\E \hat{Z}_N = 1$ and $\hat{Z}_N > 0$ almost surely, we also have
\begin{align*}
    \E \abs{\int_{[-M,M]^c} \hat{Z}_{N} \left( \beta, y N^{-1/4} \right) \rho_{\theta} (y)\, dy} =  \int_{[-M,M]^c} \rho_{\theta} (y)\, dy,
\end{align*}
which implies the uniform tail estimate
\begin{align}
    \sup_N \norm{\int_{[-M,M]^c} \hat{Z}_{N} \left( \beta, y N^{-1/4} \right) \rho_{\theta} (y)\, dy }_1 \longrightarrow 0 \quad \text{as } M \to \infty. \label{eq:g=1_pf_1}
\end{align}

To analyze $\int_{-M}^M \hat{Z}_{N}(\beta, y N^{-1/4}) \rho_{\theta} (y)\, dy$, we employ the cluster expansion approach introduced in~\cite{ALR87}. We first fix some graph-theoretic notation. Let $K_N$ be the complete graph on the vertex set $[N]$ with edge set $E_N := \{ \{i,j\} : 1 \leq i < j \leq N \}$. We identify a subgraph $\Gamma$ of $K_N$ with its edge set, so that $\Gamma \subseteq E_N$, and write $V(\Gamma) := \{i \in [N] : \{i,j\} \in \Gamma \text{ for some } j\}$ for its vertex set, $\abs{\Gamma}$ for its number of edges, and $\abs{\partial\Gamma}$ for its number of vertices of odd degree. For subgraphs $\Gamma_1, \Gamma_2$ we write $\Gamma_1 \cup \Gamma_2$ for the simple graph obtained as the union of their edge sets. A \emph{loop} $l$ is (the edge set of) a simple cycle, and a \emph{path} $p$ is a self-avoiding path. For any $\Gamma \subseteq E_N$, define
\begin{align*}
    \omega (\Gamma) := \prod_{\{i,j\} \in \Gamma} \tanh \left( \frac{\beta J_{ij}}{\sqrt{N}}\right).
\end{align*}
Since terms containing an odd power of some $\sigma_i$ vanish under $\E_{\mvgs}$, expanding the products in the definition of $\hat Z_N$ yields the cluster representation
\begin{align*}
    \hat{Z}_N (\beta, h)
    = \sum_{\Gamma \subseteq E_N} \E_{\mvgs} \prod_{\{i,j\} \in \Gamma} \left( \sigma_i \sigma_j \tanh \left( \frac{\beta J_{ij}}{\sqrt{N}} \right)\right) \prod_k \left(1 + \sigma_k \tanh h \right)
    = \sum_{\Gamma \subseteq E_N} \omega(\Gamma) \tanh^{\abs{\partial \Gamma}} (h).
\end{align*}
The strategy is to replace $\hat Z_N$ by the factorized approximation
\begin{align*}
    \tilde{Z}_{N,m} (\beta, h)
    := \prod_{l:\, \abs{l} \leq m} \big(1+ \omega(l)\big) \prod_{p:\, \abs{p} \leq m} \big(1 + \omega(p) \tanh^2 h \big),
\end{align*}
where the first product is over simple loops and the second over simple paths, and then to take logarithms. For integers $m \geq 1$ we split
\begin{align*}
    \hat{Z}_{N,m} (\beta, h)
    := \sum_{\substack{\Gamma \subseteq E_N \\ \abs{\Gamma} \leq m}} \omega(\Gamma) \tanh^{\abs{\partial \Gamma}} (h),
    \qquad
    \hat{Z}_{N,>m} (\beta, h)
    := \sum_{\substack{\Gamma \subseteq E_N \\ \abs{\Gamma} > m}} \omega(\Gamma) \tanh^{\abs{\partial \Gamma}} (h),
\end{align*}
and, with $\sum'$ denoting the sum over all collections of \emph{distinct} loops $l_1, \ldots, l_{n_1}$ and \emph{distinct} paths $p_1, \ldots, p_{n_2}$ of length at most $m$, we decompose the difference $\hat Z_N - \tilde Z_{N,m}$ as $\Delta_{N,m,k}^{(1)} + \Delta_{N,m,k}^{(2)}$, where
\begin{align*}
    \Delta_{N,m,k}^{(1)} (h)
    &:= \hat{Z}_N (\beta, h) - \sum_{n_1 + n_2 \leq k} \sum\nolimits' \prod_{i=1}^{n_1} \omega(l_i) \prod_{j=1}^{n_2} \big( \omega(p_j) \tanh^2 h \big), \\
    \Delta_{N,m,k}^{(2)} (h)
    &:= \sum_{n_1 + n_2 \leq k} \sum\nolimits' \prod_{i=1}^{n_1} \omega(l_i) \prod_{j=1}^{n_2} \big( \omega(p_j) \tanh^2 h \big) - \tilde{Z}_{N,m}(\beta, h).
\end{align*}
The contribution of large clusters is controlled by the following two lemmas.

\begin{lem}\label{lem:g=1_2}
    For $\beta < 1$ and $\theta \in \dR$,
    \begin{align*}
        \sup_N \norm{\int_{-M}^{M} \hat{Z}_{N, >m} \left( \beta, y N^{-1/4} \right) \rho_{\theta}(y)\,  dy}_2 \longrightarrow 0 \quad \text{as } m \to \infty.
    \end{align*}
\end{lem}

\begin{lem}\label{lem:g=1_3}
    For $\beta < 1$ and $\theta \in \dR$, the following hold.
    \begin{enumerate}[label=\textup{\roman*)}]
        \item $\displaystyle \norm{\int_{-M}^{M} \Delta_{N,m,k}^{(1)} \left(y N^{-1/4}\right) \rho_{\theta}(y)\, dy}_1 \leq \norm{\int_{-M}^{M} \hat{Z}_{N, >m \wedge k} \left( \beta, y N^{-1/4} \right) \rho_{\theta}(y)\,  dy}_2 + C(m,k)\, N^{-1}$ for some constant $C(m,k)$ independent of $N$.
        \item $\displaystyle \int_{-M}^{M} \Delta_{N,m_N,k_N}^{(2)} \left(y N^{-1/4}\right) \rho_{\theta}(y)\, dy \overset{d}{\longrightarrow} 0$ as $N \to \infty$, provided that $k_N \to \infty$.
    \end{enumerate}
\end{lem}

Combining Lemmas~\ref{lem:g=1_2} and~\ref{lem:g=1_3} (first letting $N\to\infty$, then $m,k\to\infty$ along a suitable diagonal), we obtain
\begin{align}
    \int_{-M}^M \hat{Z}_{N} \left( \beta, y N^{-1/4} \right) \rho_{\theta} (y)\, dy
    - \int_{-M}^M \tilde{Z}_{N,m_N} \left( \beta, y N^{-1/4} \right) \rho_{\theta} (y)\, dy
    \overset{p}{\longrightarrow} 0
    \label{eq:g=1_pf_2}
\end{align}
for some non-decreasing sequence $(m_N)_{N \geq 1}$ with $m_N \to \infty$.

The next lemma collects the limit theorems for the cluster sums needed to take the logarithm of $\tilde Z_{N,m_N}$ and to identify the limit; recall $v_1$ and $v_2$ from Theorem~\ref{thm:g=1_general}.

\begin{lem}\label{lem:g=1_etc}
    Let $\beta<1$ and let $(m_N)$ be non-decreasing with $m_N\to\infty$. The following hold as $N \to \infty$.
    \begin{enumerate}[label=\textup{\roman*)}]
        \item $\displaystyle\sum_{i<j} \ln \cosh \left(\frac{\beta J_{ij}}{\sqrt{N}} \right) - \frac{1}{2} \sum_{i<j} \tanh^2 \left(\frac{\beta J_{ij}}{\sqrt{N}} \right) \overset{p}{\longrightarrow} \frac{\beta^4}{8}\, \E J^4$
        \item $\displaystyle\sum_{l:\abs{l}\leq m_N} \abs{\ln \big( 1 + \omega (l) \big) - \omega(l) + \frac{1}{2}\,  \omega(l)^2 } \overset{p}{\longrightarrow} 0$.
        \item $\displaystyle \sup_{\abs{y} \le M} \abs{\sum_{\abs{p}\leq m_N} \ln \Big( 1 + \omega (p) \tanh^2 \big( y N^{-1/4}\big)\Big) - \frac{y^2}{\sqrt{N}} \sum_{\abs{p}\leq m_N} \omega(p) + \frac{y^4}{2N} \sum_{ \abs{p}\leq m_N} \omega(p)^2} \overset{p}{\longrightarrow} 0$.
        \item $\displaystyle\sum_{l:\text{ loops}} \omega(l)^2 \to v_1^2$ and $\displaystyle N^{-1}\sum_{p:\text{ paths}} \omega(p)^2 \to v_2^2$ in $L^2$.
        \item $\displaystyle \Big( \sum_{e:\, \text{edges}} \big(\omega(e)^2 - \E \omega(e)^2\big)\,,\; \sum_{l:\, \text{loops}} \omega(l)\,,\; N^{-1/2}\sum_{p:\, \text{paths}} \omega(p)\Big) \overset{d}{\longrightarrow} (\eta_0, \eta_1, \eta_2)$, where $\eta_0 \sim \cN \big(0, \frac{\beta^4}{2} \var ( J^2) \big)$, $\eta_1 \sim \cN \big(0, v_1^2 \big)$, and $\eta_2 \sim \cN \big(0, v_2^2 \big)$ are mutually independent, jointly with the limits in (i)--(iv).
        \item $\displaystyle\sum_{l:\, \abs{l} > m_N} \omega (l) \longrightarrow 0$ and $\displaystyle N^{-1/2}\sum_{p:\, \abs{p} > m_N} \omega (p) \longrightarrow 0$ in $L^2$; moreover $\displaystyle\sum_{l:\, \abs{l} > m_N} \omega (l)^2 \longrightarrow 0$ and $\displaystyle N^{-1}\sum_{p:\, \abs{p} > m_N} \omega (p)^2 \longrightarrow 0$ in $L^1$.
    \end{enumerate}
\end{lem}

Finally, defining
\begin{align*}
    W_N(y) :=  \sum_{l} \omega(l) - \frac{1}{2} \sum_l \omega (l)^2 + \frac{y^2}{\sqrt{N}}\sum_p \omega(p) - \frac{y^4}{2N} \sum_p \omega(p)^2,
\end{align*}
where the sums run over all simple loops and paths, the last lemma controls the tail of the limiting integral uniformly in $N$.

\begin{lem}\label{lem:g=1_5}
    For every $\eps>0$,
    \begin{align*}
        \sup_{N} \pr \left( \int_{[-M,M]^c} \exp \big(W_N (y)\big)\, \rho_{\theta} (y)\, dy > \eps \right) \longrightarrow 0 \quad \text{as } M \to \infty.
    \end{align*}
\end{lem}

\subsection{Proof of Theorem~\ref{thm:g=1_general}}\label{subsec:g=1_assembly}
\begin{proof}[Proof of Theorem~\ref{thm:g=1_general}, assuming Lemmas~\ref{lem:g=1_rho}--\ref{lem:g=1_5}]
    Taking logarithms in~\eqref{eq:g=1_Z_N},
    \begin{align}
        &N F_N(\beta,\gamma_N) - N\ln2 - \frac{N-1}{4}\beta^2 - \frac14\ln N \nonumber\\
        &= -\frac12\ln(2\pi\gamma_N)
        + \Big[ \sum_{i<j}\ln\cosh\Big(\frac{\beta J_{ij}}{\sqrt N}\Big) - \frac{N-1}{4}\beta^2 \Big] + \ln\int_\dR \hat Z_N\big(\beta, yN^{-1/4}\big) \exp\big(A_N(y)\big)\,dy .
        \label{eq:g=1_decomp}
    \end{align}
    We treat the two random terms separately, jointly in distribution.

    \emph{The prefactor.} Writing $\omega(e) = \tanh(\beta J_{ij}/\sqrt N)$ for the single edge $e = \{i,j\}$ and using $\ln\cosh = \frac12\tanh^2 + (\ln\cosh - \frac12\tanh^2)$,
    \begin{align*}
        \sum_{i<j}\ln\cosh\Big(\frac{\beta J_{ij}}{\sqrt N}\Big) - \frac{N-1}{4}\beta^2 
        &= \Big[\sum_{i<j}\ln\cosh \Big(\frac{\beta J_{ij}}{\sqrt N}\Big) - \frac12\sum_{i<j}\tanh^2 \Big(\frac{\beta J_{ij}}{\sqrt N}\Big) \Big] \\
        &\qquad + \frac12\sum_{e} \big(\omega(e)^2 - \E\omega(e)^2\big)
        + \frac12\Big[\sum_{e}\E\omega(e)^2 - \frac{(N-1)\beta^2}{2}\Big].
    \end{align*}
    By Lemma~\ref{lem:g=1_etc}(i) and (v), the first two terms converge jointly to $\frac{\beta^4}{8}\E J^4 + \frac12\eta_0$. For the deterministic term, the expansion $\E\tanh^2(\beta J/\sqrt N) = \frac{\beta^2}{N} - \frac{2\beta^4}{3N^2}\,\E J^4 + O(N^{-3})$ gives
    \begin{align*}
        \sum_{e}\E\omega(e)^2 - \frac{(N-1)\beta^2}{2}
        = \binom N2 \E\tanh^2\Big(\frac{\beta J}{\sqrt N}\Big) - \frac{(N-1)\beta^2}{2}
        \longrightarrow -\frac{\beta^4}{3}\,\E J^4 .
    \end{align*}
    Hence, jointly with the limits below,
    \begin{align}
        \sum_{i<j}\ln\cosh\Big(\frac{\beta J_{ij}}{\sqrt N}\Big) - \frac{N-1}{4}\beta^2
        \overset{d}{\longrightarrow} \frac12\,\eta_0 - \frac{\beta^4}{24}\,\E J^4 .
        \label{eq:g=1_prefactor}
    \end{align}

    \emph{The integral.} By Lemma~\ref{lem:g=1_rho} and~\eqref{eq:g=1_pf_2},
    \begin{align*}
        \int_\dR \hat Z_N\big(\beta, yN^{-1/4}\big) e^{A_N(y)}\,dy
        = \int_{-M}^M \tilde Z_{N,m_N}\big(\beta, yN^{-1/4}\big)\rho_\theta(y)\,dy
        + \int_{[-M,M]^c} \hat Z_N\,\rho_\theta\,dy + o_{P}(1),
    \end{align*}
    where the middle term is uniformly small in $L^1$ as $M\to\infty$ by~\eqref{eq:g=1_pf_1}. Taking logarithms inside the first term,
    \begin{align*}
        \ln \tilde{Z}_{N,m_N} \big( \beta, y N^{-1/4} \big)
        = \sum_{\abs{l} \leq m_N} \ln \big( 1 + \omega(l)\big) +  \sum_{\abs{p} \leq m_N} \ln \Big( 1 + \omega(p) \tanh^2\big(yN^{-1/4}\big) \Big),
    \end{align*}
    and by Lemma~\ref{lem:g=1_etc}(ii), (iii), and (vi),
    \begin{align*}
        \sup_{y \in [-M,M]} \abs{\, \ln \tilde{Z}_{N,m_N} \big( \beta, y N^{-1/4} \big) - W_N (y)\, } \overset{p}{\longrightarrow} 0 .
    \end{align*}
    Together with the tightness of $\int_{-M}^M \exp(W_N(y))\rho_\theta(y)\,dy$, which follows from Lemma~\ref{lem:g=1_etc}(iv) and (v), this yields
    \begin{align}
        \int_{-M}^M \hat{Z}_{N} \big( \beta, y N^{-1/4} \big) \rho_{\theta} (y)\, dy
        - \int_{-M}^M \exp \big(W_N (y)\big)\, \rho_{\theta} (y)\, dy \overset{p}{\longrightarrow} 0.
        \label{eq:g=1_pf_3}
    \end{align}
    Combining~\eqref{eq:g=1_pf_1},~\eqref{eq:g=1_pf_3}, and Lemma~\ref{lem:g=1_5}, and letting $M \to \infty$ after $N\to\infty$,
    \begin{align*}
        \int_{\dR} \hat{Z}_{N} \big( \beta, y N^{-1/4} \big)\, \rho_{\theta} (y)\, dy
        - \int_{\dR} \exp \big(W_N (y)\big)\, \rho_{\theta} (y)\, dy \overset{p}{\longrightarrow} 0 .
    \end{align*}
    By Lemma~\ref{lem:g=1_etc}(iv)--(vi) and the continuous mapping theorem, jointly with~\eqref{eq:g=1_prefactor},
    \begin{align*}
        \ln \int_{\dR} \exp \big(W_N (y)\big)\, \rho_{\theta} (y)\, dy
        \overset{d}{\longrightarrow}
        \eta_1 - \frac{v_1^2}{2} + \ln\int_\dR \exp\Big( \Big(\eta_2 + \frac{\theta}{2}\Big)y^2 - \Big(\frac1{12} + \frac{v_2^2}{2}\Big)y^4 \Big)\,dy .
    \end{align*}
    Inserting this and~\eqref{eq:g=1_prefactor} into~\eqref{eq:g=1_decomp}, using $\ln\gamma_N\to0$, absorbing $-\frac12\ln(2\pi)$ into the reference measure $dy/\sqrt{2\pi}$, and setting
    \begin{align*}
        \zeta_1 := \frac12\,\eta_0 + \eta_1 - \frac{\beta^4}{24}\,\E J^4 - \frac{v_1^2}{2}
        \;\sim\; \cN\Big( -\frac{\beta^4}{24}\,\E J^4 - \frac{v_1^2}{2},\; \frac{\beta^4}{8}\big(\E J^4-1\big) + v_1^2 \Big),
        \qquad \zeta_2 := \eta_2,
    \end{align*}
    completes the proof, where we used $\var(\tfrac12\eta_0) = \tfrac{\beta^4}{8}\var(J^2) = \tfrac{\beta^4}{8}(\E J^4 - 1)$ and the independence of $\eta_0, \eta_1, \eta_2$.
\end{proof}

\subsection{Proof of Lemma~\ref{lem:g=1_rho}}
\begin{proof}
    Note that $\hat{Z}_N (\beta, h) > 0$ almost surely and $\E \hat{Z}_N (\beta, h) = 1 $. Thus, we have
    \begin{align*}
        \E \abs{ \int_{\dR} \hat{Z}_{N} \left( \beta, y N^{-1/4} \right) \left\{ \exp \left( A_N(y) \right)  - \rho_{\theta}(y) \right\} dy }
        \leq \int_{\dR} \abs{ \exp \left( A_N(y) \right)  - \rho_{\theta}(y) } dy.
    \end{align*}
    By Taylor expansion, for sufficiently large $N$, we obtain
    \begin{align*}
        A_N(y) = \frac{\theta}{2} y^2 - \frac{\gamma_N^2}{12} y^4 +  \cO \left( {N}^{-1/2} \right).
    \end{align*}
    This implies the pointwise convergence
    \[A_N(y) \to \frac{\theta}{2} y^2 - \frac{1}{12} y^4.\]
    Now, we prove convergence in $L^1$. We can choose a small $\varepsilon > 0$ such that for all $\abs{x} \leq 2 \varepsilon$,
    \[\ln \cosh x \leq \frac{x^2}{2} - \frac{x^4}{24}.\]
    Since $\gamma_N \in (1/2, 2)$ for large $N$, it follows that for all $\abs{y} \leq \varepsilon N^{1/4}$,
    \[A_N( \sqrt{\gamma_N} \cdot y) \leq \frac{\theta}{2}y^2 - \frac{\gamma_N^2}{24} y^4 \leq  \frac{\theta}{2}y^2 - \frac{1}{96} y^4.\]
    Using the Dominated Convergence Theorem, we obtain
    \begin{align}
        \int_{\abs{y} \leq \varepsilon \sqrt{\gamma_N} N^{1/4}} \abs{ \exp \left( A_N(y) \right)  - \rho_{\theta}(y) } dy \to 0 \text{ as } N \to \infty. \label{eq:g=1_y^4_1}
    \end{align}
    Choose $R >0$ large enough so that for all $\abs{t} > R$,
    \[\sqrt{2} \abs{t} - \frac{t^2}{2} \leq -\frac{t^2}{4}.\]
    For large $N$ and all $\abs{t} > R$,
    \begin{align*}
        A_N( t \sqrt{\gamma_N} N^{1/4} )
        = N \left( \ln \cosh (t \sqrt{\gamma_N}) - \frac{t^2 }{2} \right)
        \leq N \left( \sqrt{2} \abs{t} - \frac{t^2}{2} \right)
        \leq -\frac{N t^2}{4}.
    \end{align*}
    Thus,
    \begin{align}
        \int_{\abs{y} \geq R \sqrt{\gamma_N} N^{1/4}}  \exp \left( A_N(y) \right) dy
        &= \sqrt{\gamma_N}\, N^{1/4} \int_{\abs{t} \geq R }  \exp \left( A_N(t \sqrt{\gamma_N} N^{1/4}) \right) dt \nonumber \\
        &\leq \sqrt{2}\, N^{1/4} \int_{\abs{t} \geq R}  \exp \left(-\frac{Nt^2}{4} \right) dt
        \to 0 \text{ as } N \to \infty. \label{eq:g=1_y^4_2}
    \end{align}
    There exists a constant $a > 0 $ such that
    \[\ln \cosh ( t ) - \frac{t^2}{2} < -a \text{ for } \varepsilon \leq \abs{t} \leq R. \]
    For $\varepsilon \leq \abs{t} \leq R$, the function $\ln \cosh \left( t \sqrt{\gamma_N} \right) - \frac{t^2}{2}$ converges uniformly to $\ln \cosh ( t ) - \frac{t^2}{2}$. Therefore, for sufficiently large $N$, we have
    \[ \ln \cosh \left( t \sqrt{\gamma_N} \right) - \frac{t^2}{2} \leq - \frac{a}{2} \text{ for } \varepsilon \leq \abs{t} \leq R. \]
    This implies that
    \[ A_N(t \sqrt{\gamma_N} N^{1/4}) \leq - \frac{aN}{2} \text{ for } \varepsilon \leq \abs{t} \leq R. \]
    Consequently,
    \begin{align}
        \int_{\varepsilon \sqrt{\gamma_N} N^{1/4} \leq \abs{y} \leq R \sqrt{\gamma_N} N^{1/4}}  \exp \left( A_N(y) \right) dy
        &\leq  N^{1/4} \int_{\varepsilon \leq \abs{t} \leq R}  \exp \left(-\frac{aN}{2} \right) dt
        \to 0 \text{ as } N \to \infty. \label{eq:g=1_y^4_3}
    \end{align}
    Combining~\eqref{eq:g=1_y^4_1},~\eqref{eq:g=1_y^4_2}, and~\eqref{eq:g=1_y^4_3} completes the proof.
\end{proof}

\subsection{Proof of Lemma~\ref{lem:g=1_2}}
\begin{proof}
    Define
    \[I_{N,m} := \E \left( \int_{-M}^{M} \hat{Z}_{N, >m} \left( \beta, y N^{-1/4} \right) \rho_{\theta}(y)\,  dy \right)^2.\]
    Since the random variables $(\omega(\Gamma))_{\Gamma\subseteq E_N}$ are pairwise orthogonal in $L^2$ (the $J_{ij}$ are independent and symmetric), a direct calculation shows that
    \begin{align*}
        I_{N,m}
        =  \sum_{\abs{\Gamma} > m} \E\, \omega(\Gamma)^2 \left( \int_{-M}^{M} \tanh^{\abs{\partial \Gamma}} \left( y N^{-1/4} \right) \rho_{\theta}(y)\,  dy \right)^2.
    \end{align*}
    Let us denote
    \begin{align*}
        \phi_N(z, y)
        := \E_{\mvgs} \prod_{i<j} \left( 1 + \sigma_i \sigma_j \tanh \left( \frac{\beta J_{ij}}{\sqrt{N}}\right) z \right) \prod_k \left(1 + \sigma_k \tanh \left( y N^{-1/4} \right) \right),
    \end{align*}
    and observe that
    \begin{align*}
        \phi_N(z, y)
        = \sum_{\Gamma \subseteq E_N} \omega(\Gamma)\, z^{\abs{\Gamma}} \tanh^{\abs{\partial \Gamma}} \left(y N^{-1/4}\right).
    \end{align*}
    Then, for $z \geq 1$, we obtain
    \begin{align*}
        I_{N,m}
        &\leq z^{-2m} \sum_{\Gamma \subseteq E_N} \E\, \omega(\Gamma)^2 z^{2 \abs{\Gamma}} \left( \int_{-M}^{M} \tanh^{\abs{\partial \Gamma}} \left( y N^{-1/4} \right) \rho_{\theta}(y)\, dy \right)^2\\
        &= z^{-2m } \E \left( \int_{-M}^{M} \phi_N(z,y) \rho_{\theta}(y)\, dy \right)^2.
    \end{align*}
    Now, we estimate $\E \big( \int_{-M}^{M} \phi_N(z,y) \rho_{\theta}(y) dy \big)^2$. Averaging over the disorder first,
    \begin{align*}
        &\E \left[\phi_N(z,u) \phi_N (z,v)\right] \\
        &=\E_{\mvgs, \mvgt} \prod_{i<j} \left( 1 + \frac{\beta_N^2 z^2}{N} \sigma_i \sigma_j \tau_i \tau_j \right) \prod_k \left(1 + \sigma_k \tanh \left(u N^{-1/4}\right) \right) \left(1 + \tau_k \tanh \left(v N^{-1/4}\right) \right),
    \end{align*}
    where $\beta_N := \big( N \E \tanh^2 \big( \beta J_{12}/\sqrt{N}\big) \big)^{1/2} < \beta$.
    Since $1 + \sigma \tanh h = e^{\sigma h} / \cosh h$ for $\sigma \in \{-1, +1\}$, and using the inequality $1+x \leq e^x$, we have
    \begin{align*}
        \E \left[\phi_N(z,u) \phi_N (z,v)\right]
        &\leq \E_{\mvgs, \mvgt} \exp \left( \frac{\beta_N^2 z^2}{N} \sum_{i<j} \sigma_i \sigma_j \tau_i \tau_j + u N^{-1/4} \sum_k \sigma_k + v N^{-1/4} \sum_k \tau_k \right) \\
        & \qquad \qquad \times \cosh^{-N} \left(u N^{-1/4} \right) \cosh^{-N} \left(v N^{-1/4}\right).
    \end{align*}
    Using the bound $\sum_{i<j} \sigma_i \sigma_j \tau_i \tau_j \leq \big( \sum_i \sigma_i \tau_i\big)^2 /2$ and applying the Hubbard--Stratonovich transform once more yields
    \begin{align*}
        \E \left[\phi_N(z,u) \phi_N (z,v)\right]
        &\leq \int \E_{\mvgs, \mvgt} \exp \left( \sum_{i} \left( \frac{\beta z w}{\sqrt{N}} \sigma_i \tau_i + u N^{-1/4} \sigma_i + v N^{-1/4} \tau_i \right)  - \frac{w^2}{2}\right) \\
        & \qquad \qquad \times \cosh^{-N} \left(u N^{-1/4}\right) \cosh^{-N} \left(v N^{-1/4}\right)  \frac{dw}{\sqrt{2 \pi}}.
    \end{align*}
    Because $\E_{\sigma, \tau} \exp (A \sigma_1 \tau_1 + B \sigma_1 + C \tau_1 ) = \cosh A \cosh B \cosh C + \sinh A \sinh B \sinh C$, we have
    \begin{align*}
        \E &\left[\phi_N(z,u) \phi_N (z,v)\right]  \\
        &\leq \int \cosh^N \left(\frac{\beta z w}{ \sqrt{N}}\right) \left( 1 + \tanh \left(\frac{\beta z w}{ \sqrt{N}}\right) \tanh \left(u N^{-1/4}\right) \tanh \left(v N^{-1/4}\right) \right)^N e^{-\frac{w^2}{2}}  \frac{dw}{\sqrt{2 \pi}}\\
        &\leq \int \left( 1 + \tanh \left(\frac{\beta z w}{ \sqrt{N}}\right) \tanh \left(u N^{-1/4}\right) \tanh \left(v N^{-1/4}\right)\right)^N \exp \left(-\frac{1 - \beta^2 z^2}{2} w^2 \right)  \frac{dw}{\sqrt{2 \pi}}.
    \end{align*}
    We can therefore deduce that, using $1+u \le e^{\abs u}$ and $\abs{\tanh x}\le\abs x$,
    \begin{align*}
        &\E \left[\phi_N(z,u) \phi_N (z,v)\right]\\
        &\leq \int \exp \left( \sqrt{N} \beta z \abs{w}\, \abs{\tanh \left(u N^{-1/4}\right) \tanh \left(v N^{-1/4}\right)} -\frac{1 - \beta^2 z^2}{2} w^2 \right)  \frac{dw}{\sqrt{2 \pi}} \\
        &\leq \frac{2}{\sqrt{1 - \beta^2 z^2}} \exp \left( \frac{N \tanh^2 \left(u N^{-1/4}\right) \tanh^2 \left(v N^{-1/4}\right) \beta^2 z^2}{2 (1 - \beta^2 z^2)}\right) \\
        &\leq \frac{2}{\sqrt{1 - \beta^2 z^2}} \exp \left( \frac{  \beta^2 z^2 }{2 (1 - \beta^2 z^2)}\, u^2 v^2\right).
    \end{align*}
    By taking $z \in \left(1, 1/\beta \right)$, we see that for all $N$,
    \begin{align*}
        \E \left[\phi_N(z,u) \phi_N (z,v)\right]
        \leq C \exp \left( C u^2 v^2\right)
    \end{align*}
    for some constant $C>0$ independent of $u,v$ and $N$. Then,
    \begin{align*}
        &\E \left( \int_{-M}^{M} \phi_N(z,y) \rho_{\theta}(y)\, dy \right)^2\\
        &\qquad\leq C \int_{[-M,M]^2}  \exp \left(C u^2 v^2 + \frac{\theta}{2} (u^2 + v^2) - \frac{1}{12} (u^4 + v^4)\right) du\, dv < \infty.
    \end{align*}
    Thus,
    \begin{align*}
        \sup_N I_{N,m} \leq z^{-2m} \sup_N \E \left( \int_{-M}^{M} \phi_N(z,y) \rho_{\theta}(y)\, dy \right)^2 \to 0 \text{ as } m \to \infty.
    \end{align*}
    This completes the proof.
\end{proof}

\subsection{Proof of Lemma~\ref{lem:g=1_3}}
\begin{proof}[Proof of (i)]
    The terms in the sum $\sum'$ where the sets of loops $l_1, \ldots, l_{n_1}$ and paths $p_1, \ldots, p_{n_2}$ are entirely vertex-disjoint exactly cancel out the corresponding term $\omega(\Gamma) \tanh^{\abs{\partial \Gamma}} \left(y N^{-1/4}\right)$, where $\Gamma = l_1 \cup \cdots \cup l_{n_1} \cup p_1 \cup \cdots \cup p_{n_2}$.
    On the other hand, for any $\Gamma$ satisfying $\abs{\Gamma} \leq k \wedge m$, there exists a collection of loops $l_1, \ldots, l_{n_1}$ and paths $p_1, \ldots, p_{n_2}$ such that $\omega(\Gamma) \tanh^{\abs{\partial \Gamma}} \left(y N^{-1/4}\right)$
    is cancelled by $\prod_{i=1}^{n_1} \omega(l_i) \prod_{j=1}^{n_2} \left( \omega(p_j) \tanh^{2} \left(y N^{-1/4}\right) \right)$.

    Consequently, $\norm{ \int_{-M}^M \Delta_{N,m,k}^{(1)} \rho_{\theta}(y) dy}_1 $ can be bounded as follows:
    \begin{align*}
        \norm{ \int_{-M}^M \Delta_{N,m,k}^{(1)} \rho_{\theta}(y) dy}_1
        &\leq \norm{\int_{-M}^M \hat{Z}_{N, >k \wedge m} \left(\beta, y N^{-1/4}\right) \rho_{\theta}(y) dy}_2 \\
        &+ \int_{-M}^M \norm{\sum_{n_1 + n_2 \leq k} \sum\nolimits'' \prod_{i=1}^{n_1} \omega(l_i) \prod_{j=1}^{n_2} \left( \omega(p_j) \tanh^{2} \left(y N^{-1/4}\right) \right)}_2 \rho_{\theta}(y)\, dy,
    \end{align*}
    where the sum $\sum''$ is taken over all collections of distinct loops $l_1, \ldots, l_{n_1}$ and paths $p_1, \ldots, p_{n_2}$ of length at most $m$ that are \emph{not} vertex-disjoint.

    We use an argument similar to~\cite[Lemma 3.1]{ALR87} to estimate the second term. Expanding the square of the $L^2$-norm, we consider pairs of configurations $l_1,l_2, \ldots, l_{n_1}$, $p_1,p_2, \ldots, p_{n_2}$ and $l_1',l_2', \ldots, l_{n_1'}'$, $p_1',p_2', \ldots, p_{n_2'}'$ such that the union $\Gamma = \bigcup l_i \cup \bigcup p_j \cup \bigcup l_i' \cup \bigcup p_j'$ contains only edges of even multiplicity, since terms containing an edge of odd multiplicity vanish in expectation by the symmetry of $J$.
    Each loop $l_i$ contributes at most $\abs{l_i}$ vertices and each path $p_j$ at most $\abs{p_j} + 1$ vertices; since each edge of $\Gamma$ is covered at least twice and the collections are not vertex-disjoint (so at least one vertex is shared across two components), we obtain the strict bound
    \begin{align*}
        2\, \abs{V(\Gamma)} < \sum_{i=1}^{n_1} \abs{l_i} + \sum_{j=1}^{n_2} (\abs{p_j} + 1 ) + \sum_{i=1}^{n_1'} \abs{l_i'} + \sum_{j=1}^{n_2'} (\abs{p_j'} + 1 ).
    \end{align*}
    We say two graphs $\Gamma, \Gamma' \subseteq E_N$ are equivalent if there is a vertex permutation mapping one to the other. For $\Gamma$ as above, we have $\abs{\Gamma} \leq 2mk$, so there is a constant $C_1 (m,k)$ such that the number of equivalence classes is at most $C_1(m,k)$. For each class, the number of embeddings of $\Gamma$ into $K_N$ is at most $N^{\abs{V(\Gamma)}}$, while each edge carries a factor $\E \tanh^{2}(\beta J/\sqrt N) = O(N^{-1})$ per pair of visits, and the number of ways to decompose a given $\Gamma$ into the specified loops and paths is bounded by a constant $C_2(m,k)$. Since the vertex-count bound above loses at least one full vertex factor relative to the disjoint case, we obtain
    \begin{align*}
        \norm{\sum_{n_1 + n_2 \leq k} \sum\nolimits'' \prod_{i=1}^{n_1} \omega(l_i) \prod_{j=1}^{n_2} \left( \omega(p_j) \tanh^{2} \left(y N^{-1/4}\right) \right)}_2
        \leq C(m,k)\, N^{-1} \sum_{i=0}^{2k} y^{2i}
    \end{align*}
    for some constant $C(m,k)$, which completes the proof.
\end{proof}

\begin{proof}[Proof of (ii)]
    Let us define
    \begin{align*}
        \psi (z) := \int_{-M}^M \prod_{\abs{l} \leq m_N} \left(1+ z \omega(l)\right) \prod_{\abs{p} \leq m_N} \left(1 + z \omega(p) \tanh^2 \left(y N^{-1/4}\right) \right) \rho_{\theta}(y)\, dy,
        \qquad z \in \bC,
    \end{align*}
    which is a polynomial in $z$ whose $k$-th order Taylor polynomial at $z=0$ is exactly $\int_{-M}^M \big(\hat Z_N - \Delta^{(1)}_{N,m_N,k}\big)\rho_\theta\,dy$ evaluated with the truncated products, so that $\int_{-M}^M \Delta^{(2)}_{N,m_N,k_N}\rho_\theta\,dy = \psi(1) - \sum_{n=0}^{k_N}\frac{1}{n!}\psi^{(n)}(0)$.
    By Cauchy's estimate for analytic functions, for any $R > 1$,
    \begin{align*}
        \abs{\psi (1) - \sum_{n=0}^{k} \frac{1}{n!} \psi^{(n)}(0)} \leq \frac{R}{R-1} \frac{1}{R^{k+1}} \sup_{\substack{z \in \bC \\ \abs{z} = R}} \abs{\psi (z)}.
    \end{align*}
    We now bound $\abs{\psi(z)}$. Observe that
    \begin{align*}
        \left| \prod_{\abs{l} \leq m_N} \right. &\left. \left(1+ z \omega(l)\right) \prod_{\abs{p} \leq m_N} \left(1 + z \omega(p) \tanh^2 \left(y N^{-1/4}\right) \right) \right|^2 \\
        &= \prod_{\abs{l} \leq m_N} \left(1+ 2 \Re(z) \omega(l) + \abs{z}^2 \omega(l)^2\right) \\
        & \qquad \times \prod_{\abs{p} \leq m_N} \left(1 + 2 \Re (z) \omega(p) \tanh^2 \left(y N^{-1/4}\right)  + \abs{z}^2 \omega(p)^2 \tanh^4 \left(y N^{-1/4}\right) \right) \\
        &\leq \exp \left( \sum_{\abs{l} \leq m_N} \left(2\Re(z) \omega(l) + \abs{z}^2 \omega(l)^2\right)\right) \\
        &\qquad \times \exp\left( \sum_{\abs{p} \leq m_N} \left( 2\Re (z) \omega(p) \tanh^2 \left(y N^{-1/4}\right)  + \abs{z}^2 \omega(p)^2 \tanh^4 \left(y N^{-1/4}\right)\right) \right).
    \end{align*}
    Thus, when $\abs{z} = R$ and $y\in[-M,M]$, using $\tanh^2(yN^{-1/4}) \le y^2 N^{-1/2}$,
    \begin{align*}
        \abs{\psi(z)}
        &\leq \exp \left( R \abs{\sum_{\abs{l} \leq m_N}  \omega(l)} + \frac{R^2}{2} \sum_{\abs{l} \leq m_N} \omega(l)^2 \right) \\
        &\quad \times \int_{-M}^M \exp \left(\left(\frac{R}{\sqrt N} \abs{\sum_{\abs{p} \leq m_N} \omega(p)} + \frac{\theta}{2} \right)y^2 + \left(\frac{R^2}{2N} \sum_{\abs{p} \leq m_N}\omega(p)^2 -\frac{1}{12} \right) y^4\right) dy.
    \end{align*}
    By Lemma~\ref{lem:g=1_etc}(iv)--(vi), the sequences 
    \[
    \sum_{\abs{l} \leq m_N}  \omega(l),\quad
    \sum_{\abs{l} \leq m_N} \omega(l)^2, \quad 
    N^{-1/2} \sum_{\abs{p} \leq m_N} \omega(p), \text{ and }N^{-1} \sum_{\abs{p} \leq m_N}\omega(p)^2
    \]
    are tight, so $\sup_{\abs z = R}\abs{\psi(z)}$ is tight for each fixed $R>1$ and $M$. Taking $k = k_N \to \infty$ in Cauchy's estimate completes the proof.
\end{proof}

\subsection{Proof of Lemma~\ref{lem:g=1_etc}}
\begin{proof}
    (i) Using $\ln \cosh \alpha = - \frac{1}{2} \ln (1 - \tanh^2 \alpha)$ and the Taylor expansion $-\frac12\ln(1-t) - \frac t2 = \frac{t^2}{4} + O(t^3)$ with $t = \tanh^2(\beta J_{ij}/\sqrt N)$,
    \begin{align*}
        \sum_{i<j} \ln \cosh \Big(\frac{\beta J_{ij}}{\sqrt{N}} \Big) - \frac{1}{2} \sum_{i<j} \tanh^2 \Big(\frac{\beta J_{ij}}{\sqrt{N}} \Big)
        = \frac14 \sum_{i<j} \tanh^4\Big(\frac{\beta J_{ij}}{\sqrt{N}} \Big) + O\Big(\sum_{i<j}\tanh^6 \Big(\frac{\beta J_{ij}}{\sqrt{N}} \Big)\Big),
    \end{align*}
    and the law of large numbers gives $\frac14\sum_{i<j}\tanh^4 \overset{p}{\longrightarrow} \frac{\beta^4}{8}\E J^4$, while the error term is $O_P(N^{-1})$.

    For (ii)--(iv) and (vi), one first verifies the following auxiliary limits as $N \to \infty$, whose proofs are similar to those in~\cite[Lemma 3.1]{ALR87} and rely on the moment assumptions on $J$.
    \begin{enumerate}[label=(\alph*)]
        \item $\pr \big( \max_{l} \abs{\omega (l)} \geq N^{\varepsilon - 3/2} \big) \to 0$ and $\pr \big( \max_{p} \abs{\omega (p)} \geq N^{\varepsilon - 1/2} \big) \to 0$ for every $\varepsilon>0$;
        \item $\sum_{l} \abs{\omega (l)}^3 \overset{p}{\longrightarrow} 0$ and $N^{-3/2}\sum_{p} \abs{\omega (p)}^3 \overset{p}{\longrightarrow} 0$, where the sums run over all loops and paths;
        \item $\sum_{\abs{l} > m_N} \omega (l) \to 0$ and $N^{-1/2}\sum_{\abs{p} > m_N} \omega (p) \to 0$ in $L^2$, and $\sum_{\abs{l} > m_N} \omega (l)^2 \to 0$ and $N^{-1}\sum_{\abs{p} > m_N} \omega (p)^2 \to 0$ in $L^1$.
    \end{enumerate}
    Item (c) is exactly (vi). For (ii), on the event in (a) we have $\abs{\omega(l)}\le\frac12$ for every loop, so $\abs{\ln(1+\omega(l)) - \omega(l) + \frac12\omega(l)^2} \le C\abs{\omega(l)}^3$, and (b) applies. For (iii), write $t_y := \tanh^2(yN^{-1/4})$ and decompose, for each path $p$,
    \begin{align*}
        &\ln\big(1+\omega(p)t_y\big) - \frac{y^2}{\sqrt N}\,\omega(p) + \frac{y^4}{2N}\,\omega(p)^2 \\
        &\qquad= \Big[\ln\big(1+\omega(p)t_y\big) - \omega(p)t_y + \frac{\omega(p)^2 t_y^2}{2}\Big]
        + \omega(p)\Big(t_y - \frac{y^2}{\sqrt N}\Big) - \frac{\omega(p)^2}{2}\Big(t_y^2 - \frac{y^4}{N}\Big).
    \end{align*}
    The first bracket is bounded by $C\abs{\omega(p)}^3 t_y^3 \le C M^6 N^{-3/2} \abs{\omega(p)}^3$ on the event in (a), so its sum vanishes by (b). Since $\abs{t_y - y^2/\sqrt N} \le C M^4 N^{-1}$ and $\abs{t_y^2 - y^4/N} \le C M^6 N^{-3/2}$ uniformly for $y\in[-M,M]$, the remaining two terms vanish by (iv) and (v), uniformly in $y \in [-M,M]$. This proves (iii).

    The first limit in (iv) follows from~\cite[Lemma 3.1]{ALR87}, and the second one can be proved similarly. Finally, the joint convergence in (v) is a direct application of~\cite[Theorem 5.1]{DW23}.
\end{proof}

\subsection{Proof of Lemma~\ref{lem:g=1_5}}
\begin{proof}
    We split $W_N (y) = W_N^{(1)} + W_N^{(2)}(y)$, where
    \begin{align*}
        W_N^{(1)} := \sum_{l} \omega(l) - \frac{1}{2} \sum_l \omega (l)^2,
        \qquad
        W_N^{(2)}(y) :=  \frac{y^2}{\sqrt{N}}\sum_p \omega(p) - \frac{y^4}{2N} \sum_p \omega(p)^2 .
    \end{align*}
    Since $\abs{y} \geq M$ implies $y^4 \geq M^2 y^2$, and $\sum_p \omega(p)^2 \geq 0$, on the event $\big\{ N^{-1/2} \sum_p \omega(p) < \frac{M^2}{12} - \frac{\theta}{2} \big\}$ we have
    \begin{align*}
        \int_{[-M,M]^c} \exp \big(W_N^{(2)}(y)\big)\, \rho_{\theta} (y)\, dy
        &\leq \int_{\dR} \exp \left( \Big(\frac{\theta}{2} + \frac{1}{\sqrt{N}} \sum_p \omega(p) - \frac{M^2}{12}\Big)\, y^2\right) dy \\
        &= \sqrt{2\pi} \left( \frac{M^2}{6} - \theta - \frac{2}{\sqrt{N}} \sum_p \omega(p)\right)^{-1/2}.
    \end{align*}
    Thus, if $\frac{M^2}{12} - \frac{\theta}{2} - \frac{2\pi}{\eps^2} > 0$, Chebyshev's inequality gives
    \begin{align*}
        \pr \left( \int_{[-M,M]^c} \exp \big(W_N^{(2)}(y)\big)\, \rho_{\theta} (y)\, dy > \eps \right)
        &\leq \pr \left( \frac{1}{\sqrt{N}} \sum_p \omega(p) > \frac{M^2}{12} - \frac{\theta}{2} - \frac{2\pi}{\eps^2}\right) \\
        &\leq \left(\frac{M^2}{12} - \frac{\theta}{2} - \frac{2\pi}{\eps^2} \right)^{-2} \frac{1}{N}\,\E \Big( \sum_p \omega(p) \Big)^2 .
    \end{align*}
    Since $\frac{1}{N}\E \big( \sum_p \omega(p) \big)^2 \to \frac{\beta^2}{2(1-\beta^2)}$ as $N \to \infty$, we obtain, for every fixed $\eps>0$,
    \begin{align*}
        \sup_N \pr \left( \int_{[-M,M]^c} \exp \big(W_N^{(2)}(y)\big)\, \rho_{\theta} (y)\, dy > \eps \right)
        \longrightarrow 0 \quad \text{as } M \to \infty.
    \end{align*}
    Finally, since $\exp(W_N) = \exp(W_N^{(1)})\exp(W_N^{(2)})$ with $W_N^{(1)}$ independent of $y$, for any $K > 0$,
    \begin{align*}
        &\sup_N \pr \left( \int_{[-M,M]^c} \exp \big(W_N(y)\big)\, \rho_{\theta} (y)\, dy > \eps \right) \\
        &\qquad \leq \sup_N \pr \left( \int_{[-M,M]^c} \exp \big(W_N^{(2)}(y)\big)\, \rho_{\theta} (y)\, dy > \frac{\eps}{K} \right)
        + \sup_N \pr \left( \exp \big(W_N^{(1)}\big) > K \right).
    \end{align*}
    The sequence $W_N^{(1)}$ is tight by Lemma~\ref{lem:g=1_etc}(iv) and (v), so the second term can be made arbitrarily small uniformly in $N$ by taking $K$ large; the first term then vanishes as $M\to\infty$ by the previous display. This completes the proof.
\end{proof}

%%%%%%%%%%%%%%%%%%%%%%%%%%%%%%%%
\section{The ferromagnetic regime: proof of Theorem~\ref{thm:ferro}}\label{sec:proof-ferro}

We work in the setting of Section~\ref{sec:mixed}. Section~\ref{subsec:coupled} records a coupled Gaussian integration-by-parts identity; Section~\ref{subsec:ferro-lemmas} assembles the concentration inputs; Section~\ref{subsec:ferro-proof} proves Theorem~\ref{thm:ferro}.

\subsection{Coupled Gaussian vectors}\label{subsec:coupled}
To analyze the fluctuations, we employ a Gaussian interpolation method.
Let $y$ be a centered Gaussian vector in $\dR^n$ with covariance matrix $C = (C_{j,j'})_{1 \le j,j' \le n}$. Consider two independent copies $y', y''$ of $y$. For $t \in [0,1]$, we define the interpolated vectors
\begin{align*}
    y_1 (t) = \sqrt{t}y + \sqrt{1-t} y' \quad \text{and} \quad y_2 (t) = \sqrt{t}y + \sqrt{1-t} y''.
\end{align*}
Let $A, B: \dR^n \to \dR$ be absolutely continuous functions such that
\begin{align*}
    \E \|\nabla A(y)\|_2^2 < \infty \quad \text{and} \quad \E \|\nabla B(y)\|_2^2 < \infty.
\end{align*}
Using Gaussian integration by parts, the derivative of the coupled expectation can be written as
\begin{align*}
    \frac{\partial}{\partial t} \E [A(y_1 (t)) B(y_2 (t))] = \sum_{j, j'=1}^n C_{j,j'} \E [\partial_{j} A(y_1 (t)) \partial_{j'} B(y_2 (t))].
\end{align*}
Applying the Fundamental Theorem of Calculus, and noting that at $t=1$ the copies coincide (both equal to $y$) while at $t=0$ they are independent (namely $y'$ and $y''$), we derive the covariance formula:
\begin{align}
    \E [A(y) B(y)] - \E A(y) \E B(y)
    = \int_0^1 \sum_{j, j'=1}^n C_{j,j'} \E [\partial_{j} A(y_1 (t)) \partial_{j'} B(y_2 (t))] \, dt.
    \label{eq:coupled_GIBP}
\end{align}
In particular, when $A=B$, the positive semi-definiteness of the covariance matrix $C$ ensures that the integrand on the right-hand side of~\eqref{eq:coupled_GIBP} is non-negative. Furthermore, using this positive semi-definiteness, we can express the integrand as
\begin{align*}
    \sum_{j, j'=1}^n C_{j,j'} \partial_{j} A(y_1 (t)) \partial_{j'} B(y_2 (t)) = \sum_{l\leq n} S_l(y_1(t)) S_l(y_2(t))
\end{align*}
for appropriately defined functions $S_l$. Consequently, this demonstrates the monotonicity of the expectations along the interpolation path.

\subsection{Concentration inputs}\label{subsec:ferro-lemmas}
Let $H_N^{1}$ and $H_N^{2}$ be two independent copies of $H_N^{mix}$. For $t \in [0,1]$, define the coupled Hamiltonians
\begin{align}
    H_{t}^{i} (\mvgs) := H_{N,t}^{i} (\mvgs) := \sqrt{t}\cdot H_N^{mix} (\mvgs) + \sqrt{1-t} \cdot H_N^{i} (\mvgs) \quad \text{for } i=1,2.
    \label{eq:H_coupled}
\end{align}
Let $\la \cdot \ra_t^{SK}$ denote the Gibbs average with respect to the product measure $\tilde{G}_t(\mvgs, \mvgt) \propto \exp(H_{t}^{1}(\mvgs) + (h+\gamma\mu) N m(\mvgs) + H_{t}^{2}(\mvgt) + (h+\gamma\mu) N m(\mvgt) )$.
The following lemma, due to Chen, establishes that the overlap of the coupled Hamiltonians concentrates in the presence of external fields; recall $u_t$ from~\eqref{eq:u_t}, here associated with $\mvgb$ and the field $h+\gamma\mu$.

\begin{lem}[{\cite[Theorem 1]{Chen13}}]\label{lem:chaos}
    For $\mvgb \in \cB$ and $0 < t < 1$, suppose $(h+\gamma\mu) \neq 0$. Then, for any $\varepsilon > 0$, we have
    \begin{align*}
    \E \la I\big(\abs{R_{\mvgs, \mvgt} - u_t} \geq \varepsilon\big)\ra_t^{SK} \leq K \exp \left( - \tfrac{N}{K}\right),
    \end{align*}
    where $K = K(\varepsilon)$ is a constant independent of $N$.
\end{lem}

Consider the coupled Hamiltonians defined in~\eqref{eq:H_coupled}.
For $i=1,2$, let $G_t^i(\mvgs)$ denote the Gibbs measure corresponding to the Hamiltonian $H_t^i (\mvgs) + h N m(\mvgs) + \frac{\gamma N}{2} m(\mvgs)^2$.
Similarly, let $G_t^{SK,i}(\mvgs)$ denote the Gibbs measure corresponding to the Hamiltonian $H_t^i (\mvgs) + (h + \gamma \mu) N m(\mvgs)$.
In other words, $\la \cdot \ra_t^{SK}$ is the expectation with respect to the product measure $G_t^{SK,1} \times G_t^{SK,2}$.
Analogously, we denote by $\la \cdot \ra_t$ and $\la \cdot \ra_t^{SKFI \times SK}$ the Gibbs averages with respect to the product measures $G_t^1 \times G_t^2$ and $G_t^{1} \times G_t^{SK,2}$, respectively.

\begin{lem}\label{lem:var}
    For any $\mvgb \in \cB$, $\gamma \geq 0$ and $\mu, h \in \dR$, we have
    \begin{align*}
        {N} \cdot \var\Big(F_N(\mvgb, \gamma, h)
        &- F_N^{SK} (\mvgb, \gamma \mu + h)\Big) \\
        &= \int_0^1 \Big( \E \la \xi(R_{\mvgs, \mvgt}) \ra_t -2 \E \la \xi(R_{\mvgs, \mvgt}) \ra_t^{SKFI \times SK} + \E \la \xi(R_{\mvgs, \mvgt}) \ra_t^{SK} \Big) dt,
    \end{align*}
    and the integrand
    \begin{align*}
    \E \la \xi(R_{\mvgs, \mvgt}) \ra_t -2 \E \la \xi(R_{\mvgs, \mvgt}) \ra_t^{SKFI \times SK} + \E \la \xi(R_{\mvgs, \mvgt}) \ra_t^{SK}
    \end{align*}
    is non-negative and non-decreasing in $t$.
\end{lem}

\begin{proof}
    Consider a Gaussian vector $y = \left( H_N^{mix}(\mvgs) \right)_{\mvgs \in \Sigma_N}$ and a function defined on $x = (x(\mvgs))_{\mvgs \in \Sigma_N}$ in $\dR^{2^N}$ by
    \begin{align*}
        A(x) &= B(x) \\
        &= \ln \sum_{\mvgs} \exp \big(x(\mvgs) + h N m(\mvgs) + \tfrac{\gamma N}{2} m(\mvgs)^2\big) - \ln \sum_{\mvgs} \exp \big(x(\mvgs) + (h+\gamma \mu) N m(\mvgs)\big).
    \end{align*}
    The result follows from a direct application of~\eqref{eq:coupled_GIBP}.
\end{proof}
The next lemma shows that the magnetization concentrates exponentially fast around $\Omega(\mvgb, \gamma, h)$.

\begin{lem}[{\cite[Proposition 1]{Chen14}}]\label{lem:magnetization}
    For any open subset $U$ of $[-1, 1]$ with $\inf \{ \abs{x -y} : x \in U, y \in \Omega(\mvgb, \gamma, h) \} > \delta$ for some $\delta > 0$, we have
    \begin{align*}
        \E \la I( m \in U ) \ra \leq K \exp \left( -{N}/{K} \right),
    \end{align*}
    where $K = K(\delta)$ is a constant independent of $N$.
\end{lem}

To proceed with the proof of Theorem~\ref{thm:ferro}, we need to control the overlap $R_{\mvgs, \mvgt}$ under the interpolated measures. Since the magnetization concentrates around $\pm \mu$ (Lemma~\ref{lem:magnetization}), the quadratic term $\frac{\gamma N}{2} m(\mvgs)^2$ essentially acts as the linear term $\gamma \mu \sum_i \sigma_i$. Motivated by this reduction, the following lemma shows that the overlaps under the coupled measures concentrate around the deterministic value $u_t$.

\begin{lem}\label{lem:overlap_concentration}
    Consider $u_t$ defined in~\eqref{eq:u_t} corresponding to $\mvgb$ and $h + \gamma \mu$.
    \begin{enumerate}[label=\textup{\roman*)}]
        \item Suppose $\Omega(\mvgb, \gamma, h) = \{\mu, -\mu\}$ for some $\mu > 0$. Then for any $0 < t < 1$ and $\varepsilon > 0$, we have, for every $N$,
        \begin{align*}
            &\E \la I\big(\big| |R_{\mvgs, \mvgt}| - u_t \big| > \varepsilon\big)\ra_t^{SKFI \times SK} \leq K_1 \exp \left( -{N}/{K_1}\right),\\
            \text{ and } 
            &\E \la I\big(\big| |R_{\mvgs, \mvgt}| - u_t \big| > \varepsilon\big)\ra_t \leq K_2 \exp \left( - {N}/{K_2}\right),
        \end{align*}
        where $K_1$ and $K_2$ are constants independent of $N$.
        \item Suppose $\Omega(\mvgb, \gamma, h) = \{\mu\}$ for some $\mu \neq 0$. Then for any $0 < t < 1$ and $\varepsilon > 0$, we have, for every $N$,
        \begin{align*}
            &\E \la I\big(\big| R_{\mvgs, \mvgt} - u_t \big| > \varepsilon\big)\ra_t^{SKFI \times SK} \leq K_1 \exp \left( - {N}/{K_1}\right),\\
            \text{ and }
            &\E \la I\big(\big| R_{\mvgs, \mvgt} - u_t \big| > \varepsilon\big)\ra_t \leq K_2 \exp \left( - {N}/{K_2}\right),
        \end{align*}
        where $K_1$ and $K_2$ are constants independent of $N$.
    \end{enumerate}
\end{lem}

\begin{proof}
    Fix $\varepsilon > 0$ and let $K$ be the constant from Lemma~\ref{lem:chaos}. Take $\delta \in (0, \sqrt{1 / K \gamma})$.

    First, we consider the case where $\Omega(\mvgb, \gamma, h) = \{\mu, -\mu\}$ for some $\mu > 0$. Since $F^{\mathrm{SK}}(\mvgb, \cdot) $ is even and convex, $F^{\mathrm{SK}} ( \mvgb, \gamma \mu + h) = F^{\mathrm{SK}} ( \mvgb, -\gamma \mu + h)$ implies that $h=0$.
    \begin{align*}
        \E \la I(
        &\big| |R_{\mvgs, \mvgt}| - u_t \big| > \varepsilon)\ra_t^{SKFI \times SK} \\
        &= \E \sum_{\mvgs, \mvgt} I( \big| |R_{\mvgs, \mvgt}| - u_t \big| > \varepsilon) G_t^1(\mvgs) G_t^{SK,2}(\mvgt) \\
        &= \E \sum_{\mvgs, \mvgt} I( \big| |R_{\mvgs, \mvgt}| - u_t \big| > \varepsilon) I(\big| |m(\mvgs)| - \mu \big| > \delta ) G_t^1(\mvgs) G_t^{SK,2}(\mvgt) \\
        &\quad + \E \sum_{\mvgs, \mvgt} I( \big| |R_{\mvgs, \mvgt}| - u_t \big| > \varepsilon) I(\big| |m(\mvgs)| - \mu \big| \leq \delta ) G_t^1(\mvgs) G_t^{SK,2}(\mvgt).
    \end{align*}
    To estimate the first term, we use Lemma~\ref{lem:magnetization}:
    \begin{align*}
        \E \sum_{\mvgs, \mvgt} &I( \big| |R_{\mvgs, \mvgt}| - u_t \big| > \varepsilon) I(\big| |m(\mvgs)| - \mu \big| > \delta ) G_t^1(\mvgs) G_t^{SK,2}(\mvgt) \\
        &\leq \E \sum_{\mvgs, \mvgt} I(\big| |m(\mvgs)| - \mu \big| > \delta ) G_t^1(\mvgs) G_t^{SK,2}(\mvgt)
        = \E \la I(\big| |m(\mvgs)| - \mu \big| > \delta )\ra
        \leq K' \exp \left( - {N}/{K'}\right).
    \end{align*}
    On the other hand, if $\big| |m(\mvgs)| - \mu \big| \leq \delta$, we have
    \begin{align*}
        G_t^1(\mvgs)
        &= \frac{\exp \big(H_t^1 (\mvgs) + \tfrac{\gamma N}{2} m(\mvgs)^2 \big)}{\sum_{\mvgr} \exp \big(H_t^1 (\mvgr) + \tfrac{\gamma N}{2} m(\mvgr)^2 \big)} \\
        &= \frac{\exp \big(H_t^1 (\mvgs) + \tfrac{\gamma N}{2}(|m(\mvgs)| - \mu)^2 - \tfrac{\gamma N}{2}\mu^2 + \gamma N \mu |m(\mvgs)| \big)}{\sum_{\mvgr} \exp \big(H_t^1 (\mvgr) + \tfrac{\gamma N}{2}(m(\mvgr) - \mu)^2 - \tfrac{\gamma N}{2} \mu^2 + \gamma N \mu m(\mvgr)\big)} \\
        &\leq \frac{\exp \big(H_t^1 (\mvgs) + \tfrac{\gamma N}{2}\delta^2 + \gamma N \mu |m(\mvgs)| \big)}{\sum_{\mvgr} \exp \big(H_t^1 (\mvgr) + \gamma N \mu m(\mvgr)\big)}
        \leq \exp \big(\gamma N\delta^2/2\big)\cdot  \big( G_t^{SK,1}(\mvgs) + G_t^{SK,1}(-\mvgs)\big).
    \end{align*}
    Thus, for the second term, using Lemma~\ref{lem:chaos},
    \begin{align*}
         \E \sum_{\mvgs, \mvgt} &I(\big| |R_{\mvgs, \mvgt}| - u_t \big| > \varepsilon) I(\big| |m(\mvgs)| - \mu \big| \leq \delta ) G_t^1(\mvgs) G_t^{SK,2}(\mvgt) \\
         &\leq 2 \exp \big(\tfrac{\gamma N}{2} \delta^2\big) \E \sum_{\mvgs, \mvgt} I(\big| |R_{\mvgs, \mvgt}| - u_t \big| > \varepsilon) G_t^{SK,1}(\mvgs) G_t^{SK,2}(\mvgt)
         \leq 2K \exp \left( - \Big(\tfrac{1}{K}-\tfrac{\gamma \delta^2}{2} \Big) N \right).
    \end{align*}
    This completes the proof of the first inequality. The second inequality follows by a similar argument.

    When $\Omega(\mvgb, \gamma, h) = \{ \mu \}$, under the assumption that $\abs{m( \mvgs) - \mu} \leq \delta$, we have
    \begin{align*}
        G_t^1(\mvgs)
        &= \frac{\exp \big(H_t^1 (\mvgs) + \tfrac{\gamma N}{2} m(\mvgs)^2 + hN m(\mvgs) \big)}{\sum_{\mvgr} \exp \big(H_t^1 (\mvgr) + \tfrac{\gamma N}{2} m(\mvgr)^2 + hN m(\mvgr) \big)} \\
        &= \frac{\exp \big(H_t^1 (\mvgs) + \tfrac{\gamma N}{2}(m(\mvgs) - \mu)^2 - \tfrac{\gamma N}{2}\mu^2 + (\gamma \mu + h) N  m(\mvgs) \big)}{\sum_{\mvgr} \exp \big(H_t^1 (\mvgr) + \tfrac{\gamma N}{2}(m(\mvgr) - \mu)^2 - \tfrac{\gamma N}{2} \mu^2 + (\gamma \mu + h) N m(\mvgr)\big)} \\
        &\leq \frac{\exp \big(H_t^1 (\mvgs) + \tfrac{\gamma N}{2}\delta^2 + (\gamma \mu + h) N m(\mvgs) \big)}{\sum_{\mvgr} \exp \big(H_t^1 (\mvgr)  + (\gamma \mu + h) N m(\mvgr)\big)} 
        = \exp \big(\gamma N\delta^2/2\big)\cdot  G_t^{\mathrm{SK},1}(\mvgs) .
    \end{align*}
    The rest of the proof follows similarly.
\end{proof}

\subsection{Proof of Theorem~\ref{thm:ferro}}\label{subsec:ferro-proof}
\begin{proof}[Proof of Theorem~\ref{thm:ferro}]
    By Lemma~\ref{lem:var}, for $t \in (0,1)$, we have
    \begin{align*}
        N & \cdot \var\big(F_N(\mvgb, \gamma, h) - F_N^{SK} (\mvgb, \gamma \mu + h)\big) \\
        &\leq t \Big[  \E \la \xi(R_{\mvgs, \mvgt}) \ra_t -2 \E \la \xi(R_{\mvgs, \mvgt}) \ra_t^{SKFI \times SK} + \E \la \xi(R_{\mvgs, \mvgt}) \ra_t^{SK} \Big] + 4(1-t) \sup_{x \in [-1, 1]} \xi(x).
    \end{align*}
    When $\Omega(\mvgb, \gamma, h) = \{-\mu, \mu \}$ for some $\mu > 0$, we have
    \begin{align*}
        \E \la \xi(R_{\mvgs, \mvgt}) \ra_t
        &= \E \la \xi(R_{\mvgs, \mvgt}) I(\big| |R_{\mvgs, \mvgt}| - u_t \big| > \varepsilon) \ra_t + \E \la \xi(R_{\mvgs, \mvgt}) I(\big| |R_{\mvgs, \mvgt}| - u_t \big| \leq \varepsilon) \ra_t\\
        &\leq K \exp \big(-\tfrac{N}{K}\big) \sup_{x \in [-1, 1]} \xi(x) + \sup_{x \in [u_t - \varepsilon, u_t + \varepsilon]} \xi(x).
    \end{align*}
    The other terms can be bounded similarly, using that $\xi$ is even so that the value $\xi(R_{\mvgs,\mvgt})$ depends only on $\abs{R_{\mvgs,\mvgt}}$. Hence,
    \begin{align*}
        \limsup_{N \to \infty} N & \cdot \var\big(F_N(\mvgb, \gamma, h) - F_N^{SK} (\mvgb, \gamma \mu+h)\big) \\
        &\leq 2t \Big[\sup_{x \in [u_t - \varepsilon, u_t + \varepsilon]} \xi(x) - \inf_{x \in [u_t - \varepsilon, u_t + \varepsilon]} \xi(x)\Big] + 4(1-t) \sup_{x \in [-1, 1]} \xi(x).
    \end{align*}
    Taking the limits $\varepsilon \to 0$ and $t \to 1$ makes the right-hand side vanish, which completes the proof for this case. The case where $\Omega(\mvgb, \gamma, h) = \{\mu \}$ for some $\mu \neq 0$ follows by a similar argument.
\end{proof}

%%%%%%%%%%%%%%%%%%%%%%%%%%%%%%%%%

\paragraph{\bfseries Acknowledgements.} The authors would like to thank Qiang Wu for helpful discussions.%, and Eric Bates and Wei-Kuo Chen for valuable comments on this paper.
\bibliographystyle{alphaurl}
\bibliography{SKFI}
\end{document}